\documentclass{amsart}
\usepackage{amscd,amsmath,amssymb,amsfonts,bbm}
\usepackage[T1]{fontenc}
\usepackage{lmodern}

\usepackage{enumerate}
\usepackage{multicol}
\usepackage[all]{xy}

\newtheorem{thm}{Theorem}[section]
\newtheorem{prop}[thm]{Proposition}
\newtheorem{lem}[thm]{Lemma}
\newtheorem{cor}[thm]{Corollary}

\theoremstyle{definition}

\newtheorem{quest}[thm]{Question}
\newtheorem{notation}[thm]{Notation}
\newtheorem{convention}[thm]{Convention}
\theoremstyle{remark}
\newtheorem{rem}[thm]{Remark}

\numberwithin{equation}{section}

\def\hq{/\hspace{-0.14cm}/}

\DeclareMathOperator{\Nef}{Nef}

 \DeclareMathOperator{\Spec}{Spec}

\DeclareMathOperator{\Ker}{Ker}
 
\DeclareMathOperator{\Pic}{Pic}

\DeclareMathOperator{\Hilb}{Hilb}
\DeclareMathOperator{\Ima}{Im}
\DeclareMathOperator{\Id}{Id}

\DeclareMathOperator{\Proj}{Proj}

\DeclareMathOperator{\Eff}{Eff}
\DeclareMathOperator{\Chow}{Chow}
\DeclareMathOperator{\mult}{mult}
\DeclareMathOperator{\thh}{th}
\DeclareMathOperator{\ic}{ci}

\DeclareMathOperator{\Eq}{E}
\DeclareMathOperator{\Cox}{Cox}
\DeclareMathOperator{\HF}{HF}

\begin{document}

\title[Geometry of the parameter space for complete intersections]
{On the birational geometry of the parameter space for codimension $2$ complete intersections}
\author{Olivier Benoist}
\address{D\'epartement de math\'ematiques et applications\\
\'Ecole normale sup\'erieure\\
45 rue d'Ulm\\
75230 Paris Cedex 05\\
France}
\email{olivier.benoist@ens.fr}

\renewcommand{\abstractname}{Abstract}
\begin{abstract}
Codimension $2$ complete intersections in $\mathbb{P}^N$ have a natural parameter space $\bar{H}$:
a projective bundle over a projective space given by the choice of the lower degree
equation and of the higher degree equation up to a multiple of the first. Motivated by the question of existence of
complete families of smooth complete intersections, we study the birational geometry of $\bar{H}$.
In a first part, we show that the first contraction of the MMP for $\bar{H}$ always
exists and we describe it. Then, we show that it is possible to run the full MMP for $\bar{H}$,
and we describe it, in two degenerate cases.
As an application, we prove the existence of complete curves in the punctual Hilbert scheme
of complete intersection subschemes of $\mathbb{A}^2$.
\end{abstract}
\maketitle


\section*{Introduction}\label{intro}

\subsection{Proper families of smooth complete intersections}

In all this paper, we work over an algebraically closed field $\mathbbm{k}$. 
It is difficult to construct interesting
complete families of smooth projective varieties over $\mathbbm{k}$.
The motivation for this paper is the following particular instance of this general problem:

\begin{quest}\label{q1}
Let $N\geq 3$ and $2\leq d_1<d_2$. Do there exist non-isotrivial complete families of
smooth complete intersections of degrees $(d_1,d_2)$ in $\mathbb{P}^N$?
\end{quest}

In order to study Question  \ref{q1}, we will parametrize these complete intersections.
Let $N\geq 1$ and $1\leq d_1<d_2$ be integers. Let $\bar{H}^{(N)}_{d_1}=\mathbb{P}(H^0(\mathbb{P}^N,\mathcal{O}(d_1)))$
be the space of degree $d_1$ hypersurfaces in $\mathbb{P}^N$,
and let $\bar{H}^{(N)}_{d_1,d_2}\to\bar{H}^{(N)}_{d_1}$ be the projective bundle whose fiber over
$\langle F\rangle$ is the projective space $\mathbb{P}(H^0(\mathbb{P}^N,\mathcal{O}(d_2))/\langle F\rangle)$
of degree $d_2$ equations up to a multiple of $F$. Points of  $\bar{H}^{(N)}_{d_1,d_2}$ will be denoted by $[F,G]$.
The supserscripts will be omitted when no confusion is possible.

Let $H_{d_1,d_2}=\{[F,G]\in\bar{H}_{d_1,d_2}|\{F=G=0\}\textrm{ is smooth of codimension }2\}$, and 
let $\Delta$ be the discriminant divisor, that is the complement
of $H_{d_1,d_2}$ in $\bar{H}_{d_1,d_2}$. 
Let $H^{\ic}_{d_1,d_2}=\{[F,G]\in\bar{H}_{d_1,d_2}|\{F=G=0\}\textrm{ has codimension }2\}$.
When $N\geq2$, $H^{\ic}_{d_1,d_2}$ (resp. $H_{d_1,d_2}$) is naturally identified with the
Hilbert scheme of complete intersections (resp. smooth complete intersections)
of degrees $d_1,d_2$ in $\mathbb{P}^N$. It will be convenient at several places
not to exclude the case $N=d_1=1$: see \ref{conv11} for the relevant conventions.

The more precise question we will be interested in is:

\begin{quest}\label{q2}
Does $H_{d_1,d_2}$ contain complete curves?
\end{quest}

When $N\geq 3$ and $d_1\geq 2$, the linear group $PGL_{N+1}$ acts properly on
$H_{d_1,d_2}$ so that the quotient $M_{d_1,d_2}=H_{d_1,d_2}/PGL_{N+1}$
exists as a separated algebraic space (see \cite{Oolsep} Corollaire 1.8): the moduli space of smooth complete intersections. 
Question \ref{q1} asks for complete curves in $M_{d_1,d_2}$: it is thus a weaker question than Question \ref{q2}.

The analogue of Question \ref{q2} for smooth hypersurfaces always has a negative answer as the corresponding discriminant
is always ample, and cannot avoid a complete curve. 
A first indication that the answer to Question \ref{q2} might be positive
is that the discriminant divisor $\Delta$ is never ample (\cite{Oolqp} Remarque 2.9). 
Then, a natural strategy to answer it is to try to  contract $\Delta$, at least birationally.
To do this, one needs to study the birational geometry of $\bar{H}_{d_1,d_2}$. More precisely, the two following
questions are relevant:

\begin{quest}\label{qMDS}
Is $\bar{H}_{d_1,d_2}$ a Mori dream space?
\end{quest}

\begin{quest}\label{qEff}
Does $\Delta$ generate an extremal ray of $\Eff(\bar{H}_{d_1,d_2})$?
\end{quest}

We refer to \cite{MDS} for the definition and basic properties of Mori dream spaces.
This roughly means that it is possible to run the minimal model program (MMP) for $\bar{H}_{d_1,d_2}$ in every direction
(\cite{MDS} Proposition 1.11).
Since it has Picard rank $2$ (it is a projective bundle over a projective space),
there are only two directions in which it is possible to run it.
One is trivial: we get the contraction
$\bar{H}_{d_1,d_2}\to\bar{H}_{d_1}$, and what we will call the MMP for $\bar{H}_{d_1,d_2}$ is the 
MMP in the other direction.

\begin{prop}
A positive answer to Questions \ref{qMDS} and \ref{qEff} would answer positively Question \ref{q2}.
\end{prop}

\begin{proof}
Since $\bar{H}_{d_1,d_2}$ is a Mori dream space, it is possible to run its MMP.
Thus, we obtain a sequence of flips, and then either a divisorial contraction to
a projective variety $X$ of Picard rank $1$, or a fibration $Y\to X$ over a Picard rank $1$ variety.

In the first case, the contracted divisor is an extremal ray of $\Eff(\bar{H}_{d_1,d_2})$: it is necessarily $\Delta$.
Take generic hyperplane sections of $X$ to get a complete curve in $X$. This curve will avoid both the
image of $\Delta$ and the flipped loci as they are of codimension $\geq 2$ in $X$.
Thus, it induces a complete curve in $H_{d_1,d_2}$, as wanted.

In the second case, the line bundle that induces the fibration is an extremal ray of
$\Eff(\bar{H}_{d_1,d_2})$: it is necessarily $\mathcal{O}(\Delta)$.
In particular, the image of $\Delta$ in $X$ is a divisor. Choose a general fiber $Y_x$ of $Y\to X$: it doesn't meet $\Delta$,
and the flipped loci have codimension $\geq 2$ in it. Take generic hyperplane sections of $Y_x$
to get a complete curve in $Y$ avoiding both $\Delta$ and the flipped loci.
It induces a complete curve in $H_{d_1,d_2}$, as wanted.
\end{proof}

\subsection{Main theorems}

We are not able to answer Questions \ref{qMDS} and \ref{qEff} in a generality that would shed light on Question \ref{q1}.
However, the goal of this paper is to give evidence for these questions. 
\vspace{1em}

In the first section of this paper, we explain why the first contraction
of the MMP for $\bar{H}_ {d_1,d_2}$ always exists, and we describe it geometrically.
We do not know in general, when this contraction is small, whether its flip exists.

 Let us be more precise. As a projective bundle over a projective
space, $\bar{H}_{d_1,d_2}$ has Picard rank $2$, and we will denote its line bundles by $\mathcal{O}(l_1,l_2)$, where $\mathcal{O}(1,0)$ comes from the base 
and $\mathcal{O}(0,1)$ is the natural relatively ample line bundle. In \cite{Oolqp}, the nef cone of $\bar{H}_{d_1,d_2}$
is shown to be generated by $\mathcal{O}(1,0)$ and $\mathcal{O}(d_2-d_1+1,1)$.
Of course, $\mathcal{O}(1,0)$ induces the projection $\bar{H}_{d_1,d_2}\to\bar{H}_{d_1}$. Although it is
not stated explicitely there, the proof in \cite{Oolqp} shows that $\mathcal{O}(d_2-d_1+1,1)$ is also semi-ample (see Proposition \ref{cont1}).
Thus, the first contraction of the MMP for 
$\bar{H}_{d_1,d_2}$ always exists. We show here that its description is as follows:

\begin{thm}[The first contraction]\label{th1}\hspace{1em}

\begin{enumerate}
\item[(i)] The locus contracted by the first contraction is $\bar{H}_{d_1-1}\times\bar{H}_{1,1}\subset \bar{H}_{d_1,d_2}$, where the inclusion is given by
$(P,[L,\Lambda])\mapsto [PL,P\Lambda^{d_2-d_1+1}]$.
\item[(ii)]
The contracted curves are exactly those in the fibers of the natural morphism
$\bar{H}_{d_1-1}\times\bar{H}_{1,1}\to\bar{H}_{d_1-1}\times\mathbb{G}(2, H^0(\mathbb{P}^N,\mathcal{O}(1)))$,
where $\mathbb{G}(2,\cdot)$ denotes the Grassmannian
of $2$-dimensional subspaces.
\end{enumerate}
\end{thm}

\vspace{1em}

In the second section, we answer positively Questions \ref{qMDS} and \ref{qEff} when $d_1=1$ and $N\geq 2$. This particular case
is not interesting from the point of view of Question \ref{q2}, since it is not difficult to construct complete families
of smooth degenerate complete intersections (see Proposition \ref{compic1}). The idea is to realize the MMP for 
$\bar{H}_{d_1,d_2}$ as a variation of GIT.

\begin{thm}[Degenerate complete intersections]\label{th2}
If $N\geq 2$ and $d_1=1$, then: 
\begin{enumerate}
\item[(i)] The variety $\bar{H}_{1,d_2}$ is a Mori dream space and its effective cone is generated by $\mathcal{O}(1,0)$ and $\Delta$. 
\item[(ii)] Unless $d_2=2$, or $N=2$ and $d_2=3$, the last step of the MMP for $\bar{H}_{1,d_2}$ is a fibration over the GIT
moduli space of degree $d_2$ hypersurfaces in $\mathbb{P}^{N-1}$.
\item[(iii)] If $d_2=2$, or $N=2$ and $d_2=3$, the last model obtained by the MMP is a compactification of $H_{1,d_2}$ with a boundary of codimension $\geq2$.
\end{enumerate}
\end{thm}

\vspace{1em}

In the third and main section of this paper, we answer positively Questions \ref{qMDS} and \ref{qEff} when $N=1$. Of course,
in this case, $H_{d_1,d_2}$ does not have an interpretation as a Hilbert scheme
of $\mathbb{P}^1$. However, to a point $[F,G]\in H_{d_1,d_2}$, it is
possible to associate the locus $\{F=G=0\}\subset\mathbb{A}^2$, realizing $H_{d_1,d_2}$ as a locally closed
subset of the punctual Hilbert scheme of $\mathbb{A}^2$. More precisely, the closure of $H_{d_1,d_2}$ in
the Hilbert scheme of $\mathbb{A}^2$ is an example of a multigraded Hilbert scheme \cite{multigraded},
that we will denote by $\hat{H}_{d_1,d_2}$.

Note that in this case, the discriminant divisor $\Delta$ is precisely
given by the classical resultant of two polynomials of degrees $d_1$ and $d_2$.

\begin{thm}[Punctual complete intersections]\label{th3}
Suppose that $N=1$. 
\begin{enumerate}
\item[(i)] The variety $\bar{H}_{d_1,d_2}$ is a
Mori dream space and its effective cone is generated by $\mathcal{O}(1,0)$ and $\Delta$. 
\item[(ii)] The MMP for $\bar{H}_{d_1,d_2}$ flips the
loci $W_i:=\{[F,G]|\deg(\gcd(F,G))\geq d_1-i\}$ for $1\leq i\leq d_1-2$ and eventually contracts $W_{d_1-1}=\Delta$.
\item[(iii)] The last model of the MMP for $\bar{H}_{d_1,d_2}$
is a compactification of $H_{d_1,d_2}$ with codimension $2$ boundary, that admits a stratification
whose normalized strata are $(H_{d_1-i,d_2+i})_{0\leq i\leq d_1-1}$.
\end{enumerate}
\end{thm}

The strategy is to show that there is a morphism $\hat{H}_{d_1,d_2}\to\bar{H}_{d_1,d_2}$,
and to give an explicit description of it as a sequence of blow-ups.
Then we construct explicit base-point free linear systems on $\hat{H}_{d_1,d_2}$
that induce birational models of $\bar{H}_{d_1,d_2}$. These birational models turn out to
realize the MMP for $\bar{H}_{d_1,d_2}$. It is worth noting that we use Theorem \ref{th1} in an essential way in the proof.
As an aside of this method, we obtain results about $\hat{H}_{d_1,d_2}$ itself:

\begin{prop}\label{thnef}
The multigraded Hilbert scheme $\hat{H}_{d_1,d_2}$ is smooth of Picard rank $d_1$.
Its nef cone is simplicial and consists exclusively of semi-ample line bundles.
\end{prop}

As a consequence of Theorem \ref{th3}, we prove the following particular case of Question \ref{q2}.
We do not know how to construct directly such curves in general (however, see Remark \ref{explicurve}).
Let us insist on the very down-to-earth content of Corollary \ref{thcc}:
it means that it is possible to
find a one-parameter algebraic family of couples $[F,G]$ of
polynomials of degrees $d_1$ and $d_2$, such that $F$ and $G$
do not acquire a common root under any degeneration.

\begin{cor}\label{thcc}
The Hilbert scheme $H_{d_1,d_2}$ of punctual complete intersections contains complete curves.
\end{cor}

\subsection{Other complete intersections}

Let us now comment on Question \ref{qMDS} in the cases that are not covered by Theorems \ref{th2} and \ref{th3}.
On the one hand, when $d_1=1$, the construction of the MMP for $\bar{H}_{d_1,d_2}$ as a variation of
GIT does not give a very explicit description
of the intermediate models. On the other hand, when $N=1$, we have a concrete description of all intermediate models, but I do
not know how to realize the MMP for $\bar{H}_{d_1,d_2}$ as a variation of GIT. Thus, none of these strategies seem to apply in general.

The general results of \cite{BCHM} (Corollary 1.3.2), that show that a log Fano variety is a
Mori dream space do not apply here, but in extremely particular cases.
The reason for it is that the MMP we are trying to run here is the traditional MMP backwards:
it worsens the nefness of the canonical bundle instead of improving it. As a consequence, the last models of the MMP for
$\bar{H}_{d_1,d_2}$, if they exist, are closer to be log Fano than $\bar{H}_{d_1,d_2}$ is.

A last strategy would be to prove that $\bar{H}_{d_1,d_2}$ is a Mori dream space by showing that
its Cox ring is finitely generated (\cite{MDS} Proposition 2.9).
This Cox ring is very easy to describe. Let $X:=H^0(\mathbb{P}^N,\mathcal{O}(d_1))\oplus H^0(\mathbb{P}^N,\mathcal{O}(d_2))$
viewed as an affine variety and $\Gamma:= H^0(\mathbb{P}^N,\mathcal{O}(d_2-d_1))$ viewed as an additive group acting on $X$ by 
$H\cdot(F,G)=(F,G+HF)$. Then $\Cox(\bar{H}_{d_1,d_2})$ is identified with the invariant ring $H^0(X,\mathcal{O}(X))^{\Gamma}$
by the natural rational map $X\dashrightarrow \bar{H}_{d_1,d_2}$.
 This gives a reformulation of Question \ref{qMDS}, and an interpretation of Theorems
\ref{th2} and \ref{th3} in the framework of Hilbert's fourteenth problem. 

\vspace{1em}

We excluded from the discussion the case $d_1=d_2$ as it is trivial, and a little bit degenerate.
Indeed, the MMP for $\bar{H}_{d_1,d_1}$ is very simple: it consists of the fibration
$\bar{H}_{d_1,d_1}\to\mathbb{G}(2,H^0(\mathbb{P}^N,\mathcal{O}(d_1)))$,
and this fibration is induced by the line bundle $\mathcal{O}(\Delta)$.
In particular, Questions \ref{q2}, \ref{qMDS} and \ref{qEff}  have a positive answer.
However, in this case, the Hilbert scheme of smooth complete intersections is not
$H_{d_1,d_1}$, but the complement of $\Delta$ in $\mathbb{G}(2,H^0(\mathbb{P}^N,\mathcal{O}(d_1)))$.
It is affine and does not contain complete curves.

\vspace{1em}

We restricted to codimension $2$ complete intersections for an explicit compactification $\bar{H}_{d_1,d_2}$ of
the Hilbert scheme  of complete intersections to exist. Under the more general condition that the degrees of the complete
intersections satisfy $d_1<d_2=\dots=d_c$, this Hilbert scheme still admits an explicit compactification that is a grassmannian bundle
over a projective space (see \cite{Oolqp} 2.1), and Questions \ref{qMDS} and \ref{qEff} still make sense and are interesting.

However, when this condition is not met, there is not such a simple compactification, and I do not know of an analogous strategy to prove
the existence of complete curves in the Hilbert scheme of smooth complete intersections, even for codimension $3$ complete intersections.

\subsection{Related works and further motivations}
The study of the birational geo\-metry of moduli spaces has recently attracted a lot of interest, for instance through
the development of the Hassett-Keel program for the moduli spaces of curves
(see \cite{HH1}, \cite{HH2}). This paper fits in this general framework.

In the particular case where $N=3$, $d_1=2$ and $d_2=3$, 
Questions \ref{qMDS} and \ref{qEff} were first asked by Casalaina-Martin,
Jensen and Laza with a motivation different from the one provided by Question \ref{q2}.
In \cite{Lazaetcie}, \cite{Lazaetcie2}, the authors are interested
in the Hassett-Keel program in genus $4$, that is in the construction of birational models
of $\bar{M}_4$ that have a modular interpretation. They construct many such birational models of $\bar{M}_4$
as GIT quotients of $\bar{H}_{2,3}$ by $PGL_4$.
A major difficulty they encounter and overcome
is that they need to apply GIT with respect to non-ample line bundles. If Question \ref{qMDS} were
known to have a positive answer, a strategy
to avoid this difficulty could have been to apply GIT with respect to genuine ample line bundles
but on the birational models of $\bar{H}_{2,3}$ appearing in its MMP.

\vspace{1em}

Another motivation for Question \ref{q2} when $N=3$ and $d_1\geq2$
is that it would give a positive answer to the following question,
that appears for instance in \cite{HM} p.57:
\begin{quest}\label{qcompcurv}
Do there exist non-trivial complete families of smooth non-dege\-ne\-rate curves in $\mathbb{P}^3$?
\end{quest}
There are obviously complete families of smooth degenerate curves in $\mathbb{P}^3$ (for instance, families of lines,
see also Proposition \ref{compic1}).
It is also well-known that there exist non-isotrivial complete families
of abstract smooth curves of genus $\geq 3$ \cite{Oortcomplete}, and that there
exist complete families of smooth non-degenerate curves in $\mathbb{P}^4$ (\cite{ChangRan1} Example 2.3).
However, by a result of Chang and Ran (\cite{ChangRan2} Theorems 1 and 3), a complete subvariety of the
Hilbert scheme of smooth non-degenerate curves in $\mathbb{P}^3$
has dimension at most $1$.

\vspace{1em}

In \cite{ABCH}, Arcara, Bertram, Coskun and Huizenga study the birational geometry of the punctual Hilbert schemes $\Hilb_n$ of length $n$
subschemes of $\mathbb{P}^2$.
This is very related to the case $N=1$ of Question \ref{qMDS} and \ref{qEff} (i.e. to Theorem \ref{th3}) because,
in this case, $H_{d_1,d_2}$ is a locally closed subscheme of $\Hilb_{d_1d_2}$.
Let us describe the similarities and differences between these two situations.

Unlike the varieties $\bar{H}_{d_1,d_2}$, $\Hilb_n$ is always log Fano (\cite{ABCH} Theorem 2.5).
This immediately implies that it is a Mori dream space by \cite{BCHM}, answering
the analogue of Question \ref{qMDS} for $\Hilb_n$.
Since $\Hilb_n$ is of Picard rank $2$, it is possible to run its MMP in two directions.
One of these is trivial: we get a contraction, the Hilbert-Chow morphism. As for $\bar{H}_{d_1,d_2}$, it is the other
one that is interesting to describe.

The analogue of Question \ref{qEff} is much more
complicated in the case of $\Hilb_n$.
Indeed, the non-trivial boundary of $\Eff(\Hilb_n)$ is difficult to describe: it depends on $n$ in a complicated and
interesting fashion (\cite{H1}, \cite{H2} Theorem 1.4 and Table 1).

Here is another difference between $\bar{H}_{d_1,d_2}$ and $\Hilb_n$.
The trivial contraction of $\bar{H}_{d_1,d_2}$ is the fibration over $\bar{H}_{d_1}$,
that associates to $\{F=G=0\}\in H_{d_1,d_2}$
its degree $d_1$ equation $F$, and the non-trivial contraction that starts the MMP for $\bar{H}_{d_1,d_2}$
is closely related to a Hilbert-Chow morphism (see for instance the proof Lemma \ref{lem1}).
On the contrary, the trivial contraction of $\Hilb_n$ is the Hilbert-Chow morphism
and the non-trivial contraction that starts the MMP for $\Hilb_n$
is given by considering the degree $n-1$ equations of a length $n$ subscheme
(\cite{ABCH} Proposition 3.1).
\bigskip

{\it Acknowledgements.} Part of this work was done during a stay at University of Utah, where I benefited from excellent working
conditions. Particular thanks to Tommaso de Fernex for many interesting discusssions and his warm hospitality.

\section{The first contraction}

In this section, we will describe the first contraction of the MMP for $\bar{H}_{d_1,d_2}$. 
Let us first recall why this contraction exists and is induced by a multiple of $\mathcal{O}(d_2-d_1+1,1)$.
It is essentially \cite{Oolqp} Th\'eor\`eme 2.7, but it is not explicitly stated there that the nef line bundle
$\mathcal{O}(d_2-d_1+1,1)$ is in fact base-point free,
and there are unneccessary additional hypotheses $N\geq 2$ and $d_1\geq 2$ in
this reference.

In all this section, a curve means an integral closed subscheme of dimension $1$.

\begin{prop}\label{cont1}
The line bundle $\mathcal{O}(d_2-d_1+1,1)$ on $\bar{H}_{d_1,d_2}$ is base-point free, but not ample.
\end{prop}

\begin{proof}
Choose a coordinate system $(X_0,\dots, X_N)$ of $\mathbb{P}^N$, and let $\mathfrak{M}_d$ be the set of monomials of degree
$d$ in the $X_s$.
Let $f^{(M)}$ (resp. $g^{(M)}$) be indeterminates indexed by $\mathfrak{M}_{d_1}$ (resp. $\mathfrak{M}_{d_2}$) and
let us work in the ring $A=\mathbbm{k}[X_s,f^{(M)}, g^{(M)}]$ trigraded by the total degree in the $X_s$, the $f^{(M)}$ and the $g^{(M)}$. 
Set $f=\sum_{M\in\mathfrak{M}_{d_1}} f^{(M)} M$ and $g=\sum_{M\in\mathfrak{M}_{d_2}}g^{(M)} M$. By \cite{Oolqp} Lemme 2.6
(i.e. by formally carrying out the euclidean division of $g$ by $f$),
there exist $q,r\in A$ homogeneous of degrees $(d_2-d_1,d_2-d_1,1)$ and
$(d_2,d_2-d_1+1,1)$ such that no monomial of $r$ is divisible by $X_0^{d_1}$ and such that:
\begin{equation}\label{euclide}
(f^{(X_0^{d_1})})^{d_2-d_1+1}g=qf+r.
\end{equation}
If $M\in\mathfrak{M}_{d_2}$, the coefficient of $M$
in $r$ is homogeneous of degree $d_2-d_1+1$ in the $f^{(M)}$ and $1$ in the $g^{(M)}$:
it induces a section in $\sigma_M\in H^0(\bar{H}_{d_1}\times\bar{H}_{d_2},\mathcal{O}(d_2-d_1+1,1))$.
Now, let $K\in H^0(\mathbb{P}^N,\mathcal{O}(d_2-d_1))$. 
Substituting the coefficients of $g+Kf$ into the coefficients of $g$ in (\ref{euclide}),
we get an identity of the form $(f^{(X_0^{d_1})})^{d_2-d_1+1}(g+Kf)=q'f+r'$.
Substracting (\ref{euclide}), we obtain $((f^{(X_0^{d_1})})^{d_2-d_1+1}K+q-q')f=r'-r$.
Since no monomial of the right-hand side is divisible by $X_0^{d_1}$, it must vanish.
This shows that if $\sigma_M$ vanishes on $(F,G)$, it also vanishes on $(F,G+KF)$:
this means that $\sigma_M$ comes from $H^0(\bar{H}_{d_1,d_2},\mathcal{O}(d_2-d_1+1,1))$ via the rational map
$\bar{H}_{d_1}\times\bar{H}_{d_2}\dashrightarrow\bar{H}_{d_1,d_2}$. 

Consider the linear system generated by the $\sigma_M$ for different
choices of coordinate systems and monomials $M$,
and let us prove that it has no base-point on $\bar{H}_{d_1,d_2}$.
If $[F,G]\in\bar{H}_{d_1,d_2}$, choose a coordinate system so that $X_0^{d_1}$
has coefficient $1$ in $F$, and substitute the coefficients of $F$ and $G$ in the
$f^{(M)}$ and $g^{(M)}$ in (\ref{euclide}) to get an identity of the form $G=QF+R$.
Since $G$ is not a multiple of $F$, there is a monomial $M$ in $R$ having non-zero
coefficient. This means exactly that $\sigma_M$ doesn't vanish on $[F,G]$.

Finally, $\mathcal{O}(d_2-d_1+1,1)$ is not ample because it has degree $0$ on some curves (\cite{Oolqp} Proposition 2.8 Etape 1).
\end{proof}

Let us denote by $c$ the contraction induced by a sufficiently large multiple of $\mathcal{O}(d_2-d_1+1,1)$, and by $\bar{c}$ the morphism
induced by the base-point free linear system used in the proof of Proposition \ref{cont1}.
Of course, $c$ and $\bar{c}$ contract the same curves: those that have intersection $0$ with $\mathcal{O}(d_2-d_1+1,1)$.
The goal of this section is to prove Theorem \ref{th1}, that describes $c$. Let us recall its statement:

\begin{thm}[Theorem \ref{th1}]\label{th1'}
\hspace{1em}

\begin{enumerate}
\item[(i)] The locus contracted by $c$ is the image of $i:\bar{H}_{d_1-1}\times\bar{H}_{1,1}\to \bar{H}_{d_1,d_2}$,
where the inclusion is given by
$i(P,[L,\Lambda])=[PL,P\Lambda^{d_2-d_1+1}]$.
\item[(ii)]
The curves contracted by $c$ are exactly those in the fibers of the morphism
$\pi:\bar{H}_{d_1-1}\times\bar{H}_{1,1}\to\bar{H}_{d_1-1}\times\mathbb{G}(2, H^0(\mathbb{P}^N,\mathcal{O}(1)))$ given by
$\pi(P,[L,\Lambda])=(P,\langle L,\Lambda\rangle)$.
\end{enumerate}
\end{thm}

The proof of Theorem \ref{th1'} will use relations between various $\bar{H}_{d_1,d_2}$: multiplication maps
$\varphi_k:\bar{H}_{d_1-k}\times \bar{H}_{k,d_2-d_1+k}\to \bar{H}_{d_1,d_2}$ defined by $\varphi_k(P,[L,H])=[PL,PH]$,
and hyperplane sections to relate $\bar{H}_{d_1,d_2}=\bar{H}^{(N)}_{d_1,d_2}$ and $\bar{H}^{(1)}_{d_1,d_2}$.
 
\begin{convention}\label{conv11}
This line of proof requires to take into account the case $N=d_1=1$. This case is degenerate, because the fibration
$\bar{H}_{1,d_2}\to\bar{H}_{1}\simeq \mathbb{P}^1$ is trivial, so that $\bar{H}_{1,d_2}\simeq \mathbb{P}^1$ has Picard rank $1$.
It will be however convenient to manipulate it as if it had Picard rank $2$. In particular,
given the definition of the line bundle $\mathcal{O}(0,1)$,
the line bundle $\mathcal{O}(l_1,l_2)$ really is another notation for the line bundle
$\mathcal{O}_{\mathbb{P}^1}(l_1-(d_2-d_1+1)l_2)$.
Moreover, in this case, the discriminant $\Delta$ is of course the empty divisor.
\end{convention} 

Before proving Theorem \ref{th1'} itself, we collect several lemmas. Lemmas \ref{lem1} and \ref{lem2} deal with the case $d_1=1$.
Lemma \ref{lem3} is technical, and is really needed only in finite characteristic (see Remark \ref{frob}). 
Lemmas \ref{lem4} and \ref{lem5} will allow a reduction to the case $N=1$.

\begin{lem}\label{lem1}
Theorem \ref{th1'} holds if $d_1=1$.
\end{lem}

\begin{proof}
Suppose first that $N\geq2$. Since $d_1=1$, if $[F,G]\in\bar{H}_{1,d_2}$, $F$ and $G$ cannot have a common factor,
so that the variety $\bar{H}_{1,d_2}=H^{\ic}_{1,d_2}$ really is the Hilbert scheme of complete intersections.
Consider the Hilbert-Chow morphism $\Psi:\bar{H}_{1,d_2}\to \Chow(\mathbb{P}^N)$. 

Let us describe the fibers of $\Psi$: choose
$[F,G]\neq [F',G']\in\bar{H}_{1,d_2}$ with the same underlying cycle $C$. If $C$ were included in only one hyperplane $H$ of $\mathbb{P}^N$,
$F$ and $F'$ would be equations of this hyperplane and would be proportional. Moreover, $G$ and $G'$ would both be equations of $C$ viewed as a
hypersurface of $H$ and would coincide up to a multiple of $F$. Thus $[F,G]=[F',G']$, which is absurd. This shows that $C$ is included in two different hyperplanes
of $\mathbb{P}^N$, so that $C$ is necessarily a codimension $2$ linear subspace with multiplicity $d_2$. Moreover the argument above also implies
that $F$ and $F'$ cannot be proportional, so that $C=d_2[\{F=F'=0\}]$. As a consequence, $[F,G]=[F,F'^{d_2}]$ and $[F',G']=[F',F^{d_2}]$.

Thus, $\Psi$ is a non-trivial contraction. Since $\bar{H}_{1,d_2}$ has Picard rank
$2$, it is the total space of exactly two non-trivial contractions:
the fibration induced by $\mathcal{O}(1,0)$, and the morphism $c$.
It follows that $\Psi=c$. The description we've just given of its
contracted locus and of its fibers is exactly
the one claimed in Theorem \ref{th1'}.

On the other hand, if $N=1$, in view of convention \ref{conv11}, $\mathcal{O}(d_2-d_1+1,1)$ is the trivial line bundle.
It thus induces the constant morphism to a point.
This is what Theorem \ref{th1'} predicts.
\end{proof}

\begin{lem}\label{lem2}
The curves contracted by $c\circ\varphi_1$ are exactly those included in the image of $i$ that are
contracted by $\pi$.
\end{lem}

\begin{proof}
These curves are exactly the curves in $\bar{H}_{d_1-1}\times\bar{H}_{1,d_2-d_1+1}$ that are contracted by
the contraction induced by the semi-ample line bundle $\varphi_1^*\mathcal{O}(d_2-d_1+1,1)$. It is easily seen
that $\varphi_1^*\mathcal{O}(d_2-d_1+1,1)=\mathcal{O}(d_2-d_1+2;d_2-d_1+1,1)$, so that this semi-ample line bundle
induces the contraction $(\Id,c)$. But in this case, $c$ has been described in Lemma \ref{lem1}.
\end{proof}

\begin{lem}\label{lem3}
Let $E\subset\bar{H}_{d_1,d_2}$ be an irreducible component of the exceptional locus of $c$.
Suppose that for every curve $C\subset E$ contracted by $c$, either $C\subset\Ima(\varphi_1)$ or $C$ meets $H^{\ic}_{d_1,d_2}$.
Then, for $[F,G]\in E$ general, $\{F=0\}$ has a reduced irreducible component.
\end{lem}

\begin{proof}
First, if there exists a curve $C\subset (E\cap\Ima(\varphi_1))$ contracted by $c$, Lemma \ref{lem2}
shows that this curve is of the form $t\mapsto[P(L+tL'),PL^{d_2-d_1+1}]$. This expression shows that for $[F,G]\in C$ general,
$F$ has a reduced irreducible component. This implies the same property for $[F,G]\in E$ general.

Suppose on the contrary that every curve $C\subset E$ contracted by $c$ meets $H^{\ic}_{d_1,d_2}$,
choose such a curve, and choose $[F,G]\in C$.
Choose a coordinate system $X_0,\dots,X_N$ of $\mathbb{P}^N$ such that $X_0^{d_1}$ appears in the expression on $F$,
and let $\rho:\mathbb{G}_m\to GL_{N+1}$ be a one-parameter
subgroup acting diagonally with carefully chosen weights $0=\alpha_0\ll\dots\ll\alpha_N$.
If the weights do not satisfy particular relations, which we assume,
$\rho$ has only finitely many fixed points on $\bar{H}_{d_1,d_2}$, namely the points
of the form $[M,M']$, where $M$ and $M'$ are monomials in the $X_i$.
In particular $\lim_{t\to 0}\rho(t)\cdot[F,G]=[X_0^{d_1},M_0']$, where $M_0'$ is a monomial.

Now consider the $1$-cycle $Z=\lim_{t\to 0}\rho(t)\cdot[C]$. First, since $E$ is $PGL_{N+1}$-invariant, $Z\subset E$.
Then, since $\mathcal{O}(d_2-d_1+1,1)$ is nef and has intersection $0$ with $C$, it has intersection $0$ with any component of $Z$.
Finally, since $Z$ is $\rho$-invariant and $\rho$ has only finitely many fixed points,
every component of $Z$ is a closure of an orbit of $\rho$.
Moreover, $[X_0^{d_1},M_0']\in Z$.
All this shows that, up to replacing $C$ by a component of $Z$ containing $[X_0^{d_1},M_0']$,
it is possible to suppose that $C$ is the closure of the orbit under
$\rho$ of a point $[F,G]\in \bar{H}_{d_1,d_2}$. Moreover, either 
$\lim_{t\to 0}\rho(t)\cdot[F,G]$ or $\lim_{t\to \infty}\rho(t)\cdot[F,G]$
is equal to $[X_0^{d_1},M_0']$, but in the second case, $\lim_{t\to 0}\rho(t)\cdot[F,G]$
is necessarily also of the form $[X_0^{d_1},\cdot]$.
Thus, up to changing the monomial $M_0'$, we may suppose that $\lim_{t\to 0}\rho(t)\cdot[F,G]=[X_0^{d_1},M_0']$.
Let us define $[M_{\infty},M_{\infty}']:=\lim_{t\to \infty}\rho(t)\cdot[F,G]$. 

Consider the map $\mathbb{P}^1\to C$ defined by
$t\mapsto\rho(t)\cdot[F,G]$: it is $\mathbb{G}_m$-equivariant with respect to the natural action on $\mathbb{P}^1$.
The line bundle $\mathcal{O}(d_2-d_1+1,1)$ has degree $0$ on $C$,
thus restricts to the trivial line bundle on $\mathbb{P}^1$.
Moreover, $\mathcal{O}(d_2-d_1+1,1)$ is naturally $GL_{N+1}$-linearized,
hence $\mathbb{G}_m$-linearized via $\rho$. Since by \cite{PicardG} Corollary 5.3, all $\mathbb{G}_m$-linearizations
of the trivial line bundle on $\mathbb{P}^1$ differ of the trivial one by
a character, it follows that the characters with which $\mathbb{G}_m$ acts on the fibers
of $\mathcal{O}(d_2-d_1+1,1)$ over $[X_0^{d_1},M_0']$ and $[M_{\infty},M_{\infty}']$ are equal.
Moreover, it is easy to calculate these characters: they are equal to 
$(d_2-d_1+1)\deg_{\alpha}(X_0^{d_1})+\deg_{\alpha}(M_0')$ and
$(d_2-d_1+1)\deg_{\alpha}(M_{\infty})+\deg_{\alpha}(M_{\infty}')$ respectively, where
$\deg_{\alpha}(X_0^{r_0}\dots X_N^{r_N}):=\sum_i\alpha_ir_i$ (see \cite{Oolqp} Proposition 2.15). Consequently,
$(d_2-d_1+1)\deg_{\alpha}(X_0^{d_1})+\deg_{\alpha}(M_0')=(d_2-d_1+1)\deg_{\alpha}(M_{\infty})+\deg_{\alpha}(M_{\infty}')$.

But if the weights have been chosen to satisfy no particular relations, this implies that
$X_0^{d_1(d_2-d_1+1)})M_0'=M_{\infty}^{d_2-d_1+1}M_{\infty}'$.
>From this equation, it follows that there exist monomials $U,V$,
where $U$ is of degree $u$ and not divisible by $X_0$,
such that $[X_0^{d_1},M_0']=[X_0^{d_1},U^{d_2-d_1+1}V]$ and
$[M_{\infty},M_{\infty}']=[X_0^{d_1-u}U,X_0^{u(d_2-d_1+1)}V]$.
We obtain a contradiction by distinguishing three cases.
If $u=0$, $[X_0^{d_1},M_0']=[M_{\infty},M_{\infty}']$, but this is only possible if we have
$[F,G]=[X_0^{d_1},M_0']=[M_{\infty},M_{\infty}']$ contradicting the fact that the orbit of $[F,G]$ is a curve.
If $u=d_1$, the expression of $M_{\infty}'$ shows that $d_1(d_2-d_1+1)\leq d_2$,
thus that $d_1=1$. But in this case, $\Ima(\varphi_1)=\bar{H}_{d_1,d_2}$, contradicting the
fact that no contracted curve is in $\Ima(\varphi_1)$. Lastly, suppose $0<u<d_1$.
Then, since $\lim_{t\to \infty}\rho(t)\cdot\langle F\rangle=\langle M_{\infty}\rangle$, the monomial
$M_{\infty}$ appears in $F$. Thus, up to modifying $G$ by a multiple of $F$,
it is possible to suppose that no monomial of $G$ is divisible by $M_{\infty}$. With this
choice of $G$, $\lim_{t\to \infty}\rho(t)\cdot\langle G\rangle=\langle M'_{\infty}\rangle$.
Note that $X_0$ divides both $M_{\infty}$ and $M'_{\infty}$. By the choice of the
weights, this implies that $X_0$ divides both $F$ and $G$,
thus that $[F,G]\notin H^{\ic}_{d_1,d_2}$. Consequently, $C$, that is the closure of the orbit of $[F,G]$, does not
meet $H^{\ic}_{d_1,d_2}$: this is again a contradiction.
\end{proof}

\begin{lem}\label{lem4}
 Let $C\subset\bar{H}_{d_1,d_2}$ be a curve contracted by $c$. Then, for $x\in\mathbb{P}^N$ general, there exists a non-zero
$\Gamma_x\in H^0(\mathbb{P}^N,\mathcal{O}(d_2))$ such that $\mult_x(\Gamma_x)\geq d_2-d_1+1$,
and such that for every $[F,G]\in C$, $\Gamma_x\in\langle F,G\rangle$.
\end{lem}

\begin{proof}
First, since $C$ is contracted by $c$ it is also contracted by $\bar{c}$.
Fix $[F,G]\in C$, let $x\in\mathbb{P}^N$ be such that
$F(x)\neq 0$, and fix a coordinate system in which $x=\{X_1=\dots=X_N=0\}$. Substitute
coefficients of $F$ and $G$ in the $f^{(M)}$ and the
$g^{(M)}$ in (\ref{euclide}) to get an identity of the form $aG=QF+R$. Since $F(x)\neq 0$,
the monomial $X_0^{d_1}$ appears in $F$, so that $a\neq 0$. Let us show that $\Gamma_x:=R$ does the job.

It is non-zero because $F \nmid G$ and $a\neq 0$. It has multiplicity $\geq d_2-d_1+1$ at $x$
because none of its monomials is divisible by $X_0^{d_1}$.  
Finally, if $[F',G']\in C$ is such that $F'(x)\neq 0$,
substitute the coefficients of $F'$ and $G'$ in the $f^{(M)}$ and the
$g^{(M)}$ in (\ref{euclide})
to get an identity of the form $a'G'=Q'F'+R'$ with $a'\neq 0$, hence $R'\neq 0$.
Since the coefficients $\sigma_M$ of $r$ in (\ref{euclide}) have been used to construct
the linear system defining $\bar{c}$, and since $[F,G]$ and $[F',G']$ have the
same image by $\bar{c}$, it follows that $R$ and $R'$ are proportional.
Thus $\Gamma_x\in\langle F',G'\rangle$. Moreover, by specialization, $\Gamma_x\in\langle F',G'\rangle$
holds in fact for every $[F',G']\in C$.
\end{proof}

\begin{lem}\label{lem5}
Suppose that $N=1$. Let $C\subset\bar{H}_{d_1,d_2}$ be a curve satisfying the following assumptions.
\begin{enumerate}
\item[(i)] If $[F,G]\in C$ is general, $F$ and $G$ do not have a common root.
\item[(ii)] If $[F,G]\in C$ is general, $F$ has a simple root.
\item[(iii)] If $[F,G],[F'G']\in C$ are general, $F$ is not proportional to $F'$.
\item[(iv)] For a general $x\in\mathbb{P}^1$ there exists a non-zero
$\Gamma_x\in H^0(\mathbb{P}^1,\mathcal{O}(d_2))$ such that $\mult_x(\Gamma_x)\geq d_2-d_1+1$,
and such that for every $[F,G]\in C$, $\Gamma_x\in\langle F,G\rangle$.
\end{enumerate}
Then $d_1=1$.
\end{lem}

\begin{proof}
Let $X_0,X_1$ be coordinates on $\mathbb{P}^1$; we will work with the inhomogeneous coordinate $X=X_1/X_0$.

First, by specialization, the assumption (iv) holds in fact for every $x\in\mathbb{P}^1$. 
Let us fix $[F,G]\in C$ general, $x\in\mathbb{P}^1$ a simple root of $F$, and consider
$\Gamma_x\in H^0(\mathbb{P}^1,\mathcal{O}(d_2))$ as in (iv).
Then there exist $Q\in H^0(\mathbb{P}^1,\mathcal{O}(d_2-d_1))$ and $\alpha\in \mathbbm{k}$ such
that $\Gamma_x=QF+\alpha G$. Since $[F,G]$ is general, by (i), $G(x)\neq 0$,
and we have $\alpha=0$. But then, since $x$ is a simple root of $F$ and $\mult_x(\Gamma_x)\geq d_2-d_1+1$, we necessarily have
$\Gamma_x=(X-x)^{d_2-d_1}F$. In particular, $(X-x)^{d_2-d_1}F\in\langle F',G'\rangle$ for every $[F',G']\in C$.

Now, let us prove that $\Pi:=\gcd_{[F,G]\in C}(F)$ is equal to $1$. Choose for contradiction a root $\pi$ of $\Pi$.
Choose $[F,G],[F'G']\in C$ general, and let $x$ be a simple root of $F$ by (ii). 
By the above, there exist $Q\in H^0(\mathbb{P}^1,\mathcal{O}(d_2-d_1))$ and
$\alpha\in \mathbbm{k}$ such that $(X-x)^{d_2-d_1}F=QF'+\alpha G'$.
Evaluating it at $\pi$, and since $G'(\pi)\neq 0$ by (i), we get $\alpha=0$.
In particular, $F'|(X-x)^{d_2-d_1}F$. It follows that there are only finitely many
possibilities for $F'$, contradicting (iii). 

Since $\Pi=1$, it is a consequence of (ii) that for
$x\in\mathbb{P}^1$ general, there exists $[F,G]\in C$ such that $x$ is a simple root of $F$.
Choose $x\in\mathbb{P}^1$ general. Let $[F,G],[F',G']\in C$
such that $x$ is a simple root of both $F$ and $F'$. I claim that $F$ and $F'$ are proportional.
Indeed, there exist $Q\in H^0(\mathbb{P}^1,\mathcal{O}(d_2-d_1))$ and $\alpha\in \mathbbm{k}$ such that $(X-x)^{d_2-d_1}F=QF'+\alpha G'$.
Since $x$ is general, it cannot be a root of both $F'$ and
$G'$ by (i). This implies $\alpha=0$ and $F'|(X-x)^{d_2-d_1}F$.
Since $F$ and $F'$ both have $x$ as a simple root, it indeed follows that they are proportional.

We will show in this paragraph that for $[F,G]\in C$ general, $F$ only has simple roots.
Let $[F,G]\in C$ be general and let $x$ be a simple root of $F$ by (ii). Suppose that $F$ has
another root $y$. Since $\Pi=1$, when $[F,G]\in C$ is chosen general, $y$ is a
general point of $\mathbb{P}^1$. As a consequence, it is possible
to find $[F',G']\in C$ such that $y$ is a simple root of $F'$. Then there exist $Q\in H^0(\mathbb{P}^1,\mathcal{O}(d_2-d_1))$ and
$\alpha\in \mathbbm{k}$ such that $(X-y)^{d_2-d_1}F'=QF+\alpha G$.
By (i), since $[F,G]\in C$ is general, $G(y)\neq 0$, and $\alpha=0$, so that $F|(X-y)^{d_2-d_1}F'$. This
shows that $x$ is a simple root of $F'$, thus that $F$ and $F'$
are proportional by the previous paragraph, and thus that $y$ is a simple root of $F$. 

It is now possible to conclude. If $d_1>1$, and $[F,G]\in C$ is general, $F$ has at least two distinct simple roots $x$ and $y$. 
Then, for general $[F',G']\in C$, there exist $Q_x,Q_y\in H^0(\mathbb{P}^1,\mathcal{O}(d_2-d_1))$ and
$\alpha_x,\alpha_y\in \mathbbm{k}$ such that $(X-x)^{d_2-d_1}F=Q_xF'+\alpha_x G'$ and $(X-y)^{d_2-d_1}F=Q_yF'+\alpha_y G'$.
Since $\Pi=1$ and $[F',G']$ is general (use (i)), we see that $\alpha_x,\alpha_y\neq 0$.
Thus $F'|(\alpha_y(X-x)^{d_2-d_1}-\alpha_x(X-y)^{d_2-d_1})F$. Again since $\Pi=1$ and $[F',G']$ is general, we get
$F'|(\alpha_y(X-x)^{d_2-d_1}-\alpha_x(X-y)^{d_2-d_1})$.
It follows by specialization that for every $[F',G']\in C$, $F'$ divides a non-trivial linear combination
of $(X-x)^{d_2-d_1}$ and $(X-y)^{d_2-d_1}$. As a consequence, $F$ itself divides a non-trivial linear combination
of $(X-x)^{d_2-d_1}$ and $(X-y)^{d_2-d_1}$. This contradicts the fact that both $x$ and $y$ are roots of $F$, and ends the proof.
\end{proof}

\begin{proof}[Proof of Theorem \ref{th1'}]
It suffices to prove that, if $C\subset \bar{H}_{d_1,d_2}$ is
a curve contracted by $c$, and if $[F,G]\in C$ is general, $\gcd(F,G)$ has degree $d_1-1$.
Indeed, this implies that $C$ is in the image of $\varphi_1$, and Lemma \ref{lem2} concludes.
We will prove this statement by induction on $d_1$, the case $d_1=1$ being trivial.
Let $C\subset \bar{H}_{d_1,d_2}$ be a curve contracted by $c$, and
let $k\in\{1,\dots,d_1\}$ be the integer such that, if $[F,G]\in C$ is general,
$\gcd(F,G)$ has degree $d_1-k$. We distinguish two cases.

If $k<d_1$, $C$ is included in the image of the multiplication map $\varphi_k$.
Let $C_1=\varphi_k^{-1}(C)$.
A calculation shows that 
$\varphi_k^*\mathcal{O}(d_2-d_1+1,1)=\mathcal{O}(d_2-d_1+2;d_2-d_1+1,1)$. Since $C\cdot\mathcal{O}(d_2-d_1+1,1)=0$,
$C_1\cdot\mathcal{O}(d_2-d_1+2;d_2-d_1+1,1)=0$, which implies that the image $C_2$ of $C_1$ in $\bar{H}_{k,d_2-d_1+k}$ is a curve satisfying
$C_2\cdot\mathcal{O}(d_2-d_1+1,1)=0$. Hence it is a curve contracted by $c$. Moreover, by the choice of $k$, if $[F,G]\in C_2$ is general,
$\gcd(F,G)$ has degree $0$. By the induction hypothesis, this implies $k=1$, as wanted.

Suppose now that $k=d_1$; we need to prove that $d_1=1$.
Let $E$ be an irreducible component of the exceptional locus of $c$ containing $C$.
Applying Lemma \ref{lem3}, it is possible, up to changing $C$, to suppose
that for $[F,G]\in C$ general, $\{F=0\}$ has a reduced irreducible component.
By Lemma \ref{lem4}, for every $x\in\mathbb{P}^N$ there exists a non-zero
$\Gamma'_x\in H^0(\mathbb{P}^N,\mathcal{O}(d_2))$ such that
$\mult_x(\Gamma'_x)\geq d_2-d_1+1$, and such that for every $[F,G]\in C$, $\Gamma'_x\in\langle F,G\rangle$.
Choose a general linear subspace $\mathbb{P}^1\subset\mathbb{P}^N$.
If $[F,G]\in C$ is a general point, one obtains by restriction to $\mathbb{P}^1$
a point $[F,G]\in\bar{H}^{(1)}_{d_1,d_2}$: this induces a curve
$C'\subset\bar{H}^{(1)}_{d_1,d_2}$. If $x\in\mathbb{P}^1$, one obtains by restricting
$\Gamma'_x$ an element $\Gamma_x\in H^0(\mathbb{P}^1,\mathcal{O}(d_2-d_1+k))$.
Let us check that the hypotheses of Lemma \ref{lem5} are satisfied.
All are immediate consequences of the genericity of the chosen subspace $\mathbb{P}^1\subset\mathbb{P}^N$,
and of an additional argument. For (i), you need to use the fact that $k=d_1$, (ii) is a consequence of the fact that
for $[F,G]\in C$ general, $\{F=0\}$ has a reduced irreducible component,
and (iv) is deduced from the corresponding properties of $\Gamma'_x$. As for (iii), note that if
$[F,G],[F',G']\in C$ are general, $F$ cannot be proportional to $F'$
: if it were the case, $C$ would be in a fiber of the projection $\bar{H}_{d_1,d_2}\to\bar{H}_{d_1}$
and its intersection with $\mathcal{O}(d_2-d_1+1,1)$ would be positive.
Then Lemma \ref{lem5} applies and shows that $d_1=1$, as wanted.
\end{proof}

\begin{rem}\label{frob}
Lemma \ref{lem5} would remain true in characteristic $0$ without the hypothesis (ii), making the use Lemma \ref{lem3} unnecessary.
However, in finite characteristic, it is not the case. As an example,
in finite characteristic $p$, the curve $C\subset \bar{H}_{2,4}$ defined by
$t\mapsto [(X+t)^p,X^{2p}]$ satisfies all assumptions but (ii) of Lemma \ref{lem5}
(take $\Gamma_x=(X-x)^{2p}$), but not its conclusion. 
\end{rem}

\section{Degenerate complete intersections}

In this section, we keep the previous notations, but we set $d_1=1$ and suppose that $N\geq 2$.
Then, the whole of $\bar{H}_{d_1,d_2}$ is the Hilbert scheme $H^{\ic}_{d_1,d_2}$ of complete
intersections, as $F$ and $G$ cannot have a common factor.

In this case, it is not difficult to construct complete families of smooth complete intersections.
Note that since the moduli space of smooth hypersurfaces in $\mathbb{P}^{N-1}$ is affine (\cite{GIT} Proposition 4.2),
such families are necessarily isotrivial.

\begin{prop}\label{compic1}
There exist complete curves in $H_{1,d_2}$.
\end{prop}

\begin{proof}
Fix $H\subset\mathbb{P}^{N-1}$ an arbitrary smooth hypersurface of degree $d_2$.
The set of embeddings of $\mathbb{P}^{N-1}$ in $\mathbb{P}^{N}$ is naturally identified with
an open subset of the space $\mathbb{P}(M_{N+1,N})$ of matrices up to scalar. Moreover,
its complement has codimension $\geq 2$, as it is defined by the vanishing of several minors.
By taking generic hyperplane sections of $\mathbb{P}(M_{N+1,N})$, one obtains a complete curve included in the space of embeddings
of $\mathbb{P}^{N-1}$ in $\mathbb{P}^{N}$. Considering the image of $H$ by these embeddings,
we get a complete family of smooth complete intersections.
\end{proof}

The trick used in this proof will allow us to realize the MMP of $\bar{H}_{1,d_2}$ as a variation of GIT (see \cite{VGIT}).
Let us introduce the space $X=\bar{H}_{d_2}^{(N-1)}\times\mathbb{P}(M_{N+1,N})$.
The linear group $G=SL_{N}$ acts diagonally by $g\cdot(F,M)=(F\circ g^{-1},Mg^{-1})$,
and all line bundles on $X$ are naturally $G$-linearized.

\begin{prop}\label{VGIT}
\hspace{1em}

\begin{enumerate}
\item[(i)] $X\hq_{\mathcal{O}(0,1)}G=\bar{H}_1^{(N)}$.
\item[(ii)] $X\hq_{\mathcal{O}(\varepsilon,1)}G=\bar{H}_{1,d_2}$ if $0<\varepsilon<\frac{1}{d_2(N-1)}$.
\item[(iii)] $X\hq_{\mathcal{O}(1,0)}G=\bar{H}_{d_2}^{(N-1)}\hq G$ is the
GIT moduli space $\mathcal{H}$ of degree $d_2$ hypersurfaces in $\mathbb{P}^{N-1}$.
\item[(iv)] If $d_2=2$, or $N=2$ and $d_2=3$, and if $0<\varepsilon\ll 1$,
$X\hq_{\mathcal{O}(1,\varepsilon)}G$ is a compactification of $H_{1,d_2}$ with
a boundary of codimension $\geq 2$.
\end{enumerate}
\end{prop}

\begin{proof}
By functoriality of GIT,
$X\hq_{\mathcal{O}(0,1)}G=\mathbb{P}(M_{N,N+1})\hq_{\mathcal{O}(1)}G$. Let us show that the rank $N$ matrices are stable
and that the other matrices are unstable. This will imply (i) because
$\mathbb{P}(M_{N+1,N})\hq_{\mathcal{O}(1)}G$ is then the geometric quotient of the open set
$U$ of rank $N$ matrices by $G$, and because this geometric quotient is the map $U\to \bar{H}_1^{(N)}$ that associates to a matrix its image.
To check it, we use the Hilbert-Mumford criterion (\cite{GIT} Chapter 2 Theorem 2.1).
First, if $M$ is a matrix of rank $<N$, let us choose a basis of $\mathbb{P}^{N-1}$ such that the
last column of $M$ is zero. Then consider the one-parameter subgroup $\lambda$ of $G$ acting
diagonally on this basis with weights $(-1,\dots,-1,N-1)$. A simple calculation shows that
$\mu^{\mathcal{O}(1)}(M,\lambda)=-1<0$, so that $M$ is unstable. Conversely,
if $M$ is a rank $N$ matrix, and $\lambda$ is any non-trivial one-parameter subgroup of $G$, choose a basis
of $\mathbb{P}^{N-1}$ such that $\lambda$ acts diagonally with weights $\lambda_1\leq\dots\leq\lambda_N$;
those weights are not all zero and add up to zero.
A calculation shows that $\mu^{\mathcal{O}(1)}(M,\lambda)=\lambda_N>0$, so that $M$ is indeed stable.

Let us use the Hilbert-Mumford criterion as above to show that
$X^s(\mathcal{O}(\varepsilon,1))=X^{ss}(\mathcal{O}(\varepsilon,1))=\bar{H}_{d_2}^{(N-1)}\times U$.
This implies that $X\hq_{\mathcal{O}(\varepsilon,1)}G$ is
the geometric quotient of $\bar{H}_{d_2}\times U$ by $G$, but this geometric quotient is the
morphism $\bar{H}_{d_2}^{(N-1)}\times U\to \bar{H}_{1,d_2}$ given by $(F,M)\mapsto M(\{F=0\})$, proving (ii). 
First, if $(F,M)$ is such that $M$ has rank $<N$, let us choose a basis of $\mathbb{P}^{N-1}$ such that the
last column of $M$ is zero. Then consider the one-parameter subgroup $\lambda$ of $G$ acting diagonally on this basis
with weights $(\lambda_1,\dots,\lambda_N)=(-1,\dots,-1,N-1)$.
If we denote $\deg_{\lambda}(F)$ the weighted degree of $F$, that is the maximum over
the monomials $M=X_1^{r_1}\dots X_N^{r_N}$ appearing in $F$ of the quantities $\sum_i \lambda_i r_i$, an easy calculation shows:
$\mu^{\mathcal{O}(\varepsilon,1)}((F,M),\lambda)=-1+\varepsilon\deg_{\lambda}(F)\leq 1+\varepsilon d_2(N-1)<0$. Thus, $(F,M)$ is unstable.
Conversely, if $(F,M)$ is such that $M$ has rank $N$ and $\lambda$ is any non-trivial one-parameter subgroup of $G$, choose a basis
of $\mathbb{P}^{N-1}$ such that $\lambda$ acts diagonally with weights $\lambda_1\leq\dots\leq\lambda_N$;
those weights are not all zero and add up to zero.
Then $\mu^{\mathcal{O}(\varepsilon,1)}((F,M),\lambda)=\lambda_N+\varepsilon\deg_{\lambda}(F)\geq\lambda_N+\varepsilon d_2\lambda_1
=\lambda_N-\varepsilon d_2(\lambda_2+\dots+\lambda_N)\geq\lambda_N(1-\varepsilon d_2(N-1))>0$. Thus, $(F,M)$ is stable.

Part (iii) is an immediate consequence of functoriality of GIT.

Let us conclude by proving (iv). If $(F,M)\in X$ is such that $\{F=0\}$ is smooth and $M$ has rank $N$,
$(F,M)\in X^s(\mathcal{O}(\varepsilon,1))$ by (ii)
and $(F,M)\in X^{ss}(\mathcal{O}(1,0))$ by (iii) and because smooth hypersurfaces
are GIT-semi-stable (\cite{GIT} Proposition 4.2). As a consequence,
$(F,M)\in X^s(\mathcal{O}(1,\varepsilon))$, and the geometric quotient
of this locus, that is $H_{1,d_2}$, is an open subset of $X\hq_{\mathcal{O}(1,\varepsilon)}G$.
On the other hand, if $\{F=0\}$ is singular, $(F,M)\notin X^{ss}(\mathcal{O}(1,0))$.
Indeed, since $d_2=2$, or $N=2$ and $d_2=3$, $G$ acts transitively on
$\bar{H}_{d_2}^{(N-1)}\setminus\Delta$, so that $G$-invariant divisors on
$\bar{H}_{d_2}^{(N-1)}$ are necessarily multiples of $\Delta$, and all singular hypersurfaces
are unstable.  It follows that $(F,M)\notin X^{ss}(\mathcal{O}(1,\varepsilon))$ if $\{F=0\}$ is singular. 
Consequently,
the complement of $\{(F,M)\in X| \{F=0\}\textrm{ is smooth and } M\textrm{ has rank }N\}$
in $X^{ss}(\mathcal{O}(1,\varepsilon))$ has codimension $\geq 2$ because 
the condition for a matrix to be of rank $<N$ is given by the vanishing of several minors.
Looking at the image in the quotient, this shows that the complement of
$H_{1,d_2}$ in $X\hq_{\mathcal{O}(1,\varepsilon)}G$ has codimension $\geq 2$.
\end{proof}


As a consequence, Theorem \ref{th2} follows:

\begin{thm}[Theorem \ref{th2}]\label{th2'}
If $N\geq 2$ and $d_1=1$, then: 
\begin{enumerate}
\item[(i)] The variety $\bar{H}_{1,d_2}$ is a Mori dream space and its effective cone is generated by $\mathcal{O}(1,0)$ and $\Delta$. 
\item[(ii)] Unless $d_2=2$, or $N=2$ and $d_2=3$, the last step of the MMP for $\bar{H}_{1,d_2}$ is a fibration over the GIT
moduli space $\mathcal{H}$ of degree $d_2$ hypersurfaces in $\mathbb{P}^{N-1}$.
\item[(iii)] If $d_2=2$, or $N=2$ and $d_2=3$, the last model obtained by the MMP is a
compactification of $H_{1,d_2}$ with a boundary of codimension $\geq2$.
\end{enumerate}
\end{thm}

\begin{proof}
The variety $\bar{H}_{1,d_2}$ is a GIT quotient of a Mori dream space by Proposition \ref{VGIT} (ii). It follows that it
is a Mori dream space by\cite{GITMDS} Theorem 1.1. Moreover,
the general theory of variation of GIT \cite{VGIT} shows that the GIT quotients of $X$ by $G$ when the polarization varies
fit together to form a sequence of flips and contractions, realizing the MMP for $\bar{H}_{1,d_2}$.

Let us distinguish two cases. Suppose that we do not have $d_2=2$, or $N=2$ and $d_2=3$, so that $\dim(\mathcal{H})>0$.
Then, $\mathcal{H}$ is the last model
of the  MMP for $\bar{H}_{1,d_2}$ (Proposition \ref{VGIT} (iii)). In particular, since $\dim(\mathcal{H})<\dim(\bar{H}_{1,d_2})$,
the last step of this MMP is a fibration. Moreover, since 
$\Delta$ is the pull-back of the discriminant of $\mathcal{H}$ by the rational map $\bar{H}_{1,d_2}\dashrightarrow \mathcal{H}$,
this fibration is induced by $\mathcal{O}(\Delta)$. This shows that $\Delta$ is an extremal ray of
the effective cone of $\bar{H}_{1,d_2}$, the other
ray being obviously $\mathcal{O}(1,0)$.

Suppose on the contrary that $d_2=2$, or $N=2$ and $d_2=3$: in these cases, $\mathcal{H}$ is a point.
Then the VGIT still realizes the MMP for $\bar{H}_{1,d_2}$ but the last model of this MMP is now $X\hq_{\mathcal{O}(1,\varepsilon)}G$
for $0<\varepsilon\ll 1$ (Proposition \ref{VGIT} (iv)). Since $X\hq_{\mathcal{O}(1,\varepsilon)}G$ is a compactification of $H_{1,d_2}$
with a boundary of codimension $\geq2$,
the last step of this MMP is a divisorial contraction contracting $\Delta$.
This shows that $\Delta$ is an extremal ray of the effective cone of $\bar{H}_{1,d_2}$, the other
ray being obviously $\mathcal{O}(1,0)$.
\end{proof}

\begin{rem}
This construction of the MMP for $\bar{H}_{1,d_2}$ does not allow to obtain an explicit description of all intermediate models
(for instance, it is difficult in general to describe the GIT-stable hypersurfaces).
However, the reader may check what follows as an exercise in GIT.

The union of the flipped loci in $\bar{H}_{1,d_2}$ is the set of complete intersections
that are GIT-unstable as hypersurfaces in $\mathbb{P}^{N-1}$.
The flipped loci are unions of strata of the Hesselink stratification
of this unstable locus (see \cite{Hesselink} Paragraph 6 or \cite{Kirwan} Chapter 12).
In the particular case when $N=2$ and $d_2=3$, the MMP first flips the locus of complete
intersections supported on one single point, and then contracts $\Delta$.
When $d_2=2$, the MMP first flips the locus of quadrics that are double linear spaces,
then the locus of quadrics that are union of two linear spaces, then successively
the loci of quadrics of higher and higher rank, until it contracts $\Delta$.
\end{rem}

\begin{rem}
Suppose that $d_2=2$, and that either the characteristic is not $2$ or that $N$ is even.
Then it is easy to construct by hand the compactification of $H_{1,2}$
with small boundary that is the last step of the MMP. Indeed, the dual of
a smooth complete intersection is then a quadric cone, and
duality induces a rational map $\bar{H}_{1,2}\dashrightarrow \bar{H}_2^{(N)}$
that realizes an isomorphism between $H_{1,2}$ and the set of rank $N$ quadrics.
The required compactification is the set of quadrics of rank $\leq N$ in $\bar{H}_2^{(N)}$. 

However, this construction doesn't work in characteristic $2$ when $N$ is odd, due to the bad behaviour of duality:
the dual of a smooth complete intersection in this case is a double hyperplane.
\end{rem}

\section{Punctual complete intersections}\label{part3}
In this section, we set $N=1$. As it will be important to take into account
the case $d_1=1$, keep in mind the conventions made in \ref{conv11}.

The class of the discriminant in $\Pic(\bar{H}_{d_1,d_2})$ has been calculated in general in
\cite{Ooldeg} Exemple 1.11. When $N=1$, this specializes to the classical formula for the
degrees of the resultant, that we recall for later use: $\mathcal{O}(\Delta)=\mathcal{O}(d_2,d_1)$.

\subsection{Blowing-up $\bar{H}_{d_1,d_2}$}

Here, we will construct and describe a suitable blow-up of $\bar{H}_{d_1,d_2}$. 
For $1\leq k\leq d_1-1$, we consider the multiplication map
$\varphi_k:\bar{H}_{d_1-k}\times\bar{H}_{k,d_2-d_1+k}\to\bar{H}_{d_1,d_2}$ defined by
$\varphi_k(P,[L,H])=[PL,PH]$. We denote by $W_k$ the image of $\varphi_k$ with its reduced structure. In particular, $W_{d_1-1}=\Delta$.
Let $\hat{H}_{d_1,d_2}$ be the scheme obtained by blowing up first $W_1$, then
the strict transform of $W_2$, $ \dots $, and lastly the strict transform of $W_{d_1-1}$.
Let $E_1,\dots, E_{d_1-1}$ be
the exceptional divisors of these blow-ups.

The fact, claimed in the introduction,
that $\hat{H}_{d_1,d_2}$ might have been defined as
the closure of $H_{d_1,d_2}$ in the appropriate Hilbert scheme will only be proven
in the last paragraph \ref{idhilb} of this section.

\begin{notation}
Note that the dependence on $d_1$ and $d_2$ of $\varphi_k$ is not explicit in the notation.
The context will always make clear what morphism is intended.
Moreover, we will still denote by $\varphi_k$ morphisms induced by $\varphi_k$
after some blow-ups have been performed.
A similar remark holds for the loci $W_k$ and $E_k$: their strict transforms
will still be denoted by $W_k$ and $E_k$ after some blow-ups have been performed.

In a similar abuse of notation, we will still write $\mathcal{O}(l_1,l_2)$ for
the pull-back of $\mathcal{O}(l_1,l_2)$ on any blow-up of $\bar{H}_{d_1,d_2}$.

It will be sometimes easier to work on $\bar{H}_{d_1}\times\bar{H}_{d_2}$ instead
of $\bar{H}_{d_1,d_2}$. For this reason, we introduce the morphisms 
$\tilde{\varphi}_k:\bar{H}_{d_1-k}\times\bar{H}_{k}\times\bar{H}_{d_2-d_1+k}\to\bar{H}_{d_1}\times\bar{H}_{d_2}$
defined by $\tilde{\varphi}_k(P,L,H)=(PL,PH)$, and the loci $\tilde{W}_k=\Ima(\tilde{\varphi}_k)$. Notice that
$\tilde{W}_0=\{(F,G)\mid F|G \}$. By convention, $W_0=\varnothing$.

The blow-up of $W_k$ in a space $X$ will be denoted by $\beta_k:\beta_kX\to X$.
Moreover, the notation, $\beta_l^k$ will denote $\beta_k\dots\beta_l$. For instance,
$\hat{H}_{d_1,d_2}=\beta_1^{d_1-1}\bar{H}_{d_1,d_2}$.

Finally, to shorten notations, we will write $S_l$ instead of $H^0(\mathbb{P}^1,\mathcal{O}(l))$. 
\end{notation}

The goal of this paragraph is to prove:

\begin{prop}\label{blowup}
\hspace{1em}

\begin{enumerate}
\item[(i)] The variety $\hat{H}_{d_1,d_2}$ is smooth and the $(E_k)_{1\leq k\leq d_1-1}$ form a strict normal crossing divisor in it.
\item[(ii)] For $1\leq k\leq d_1-1$, there is a natural isomorphism $E_k\simeq \hat{H}_{d_1-k,d_2+k}\times\hat{H}_{k,d_2-d_1+k}$.
The two natural projections will be denoted by $p_1$ and $p_2$.
\item[(iii)] If $j<k$, $E_j|_{E_k}=p_2^*E_j$.
\item[(iv)] If $j>k$, $E_j|_{E_k}=p_1^*E_{j-k}$.
\item[(v)] $\mathcal{O}(E_k)|_{E_k}=p_1^*\mathcal{O}(1,-1)\otimes p_2^*\mathcal{O}(1,1)(-E_1-\dots-E_{k-1})$.
\item[(vi)] $\mathcal{O}(l_1,l_2)|_{E_k}=p_1^*\mathcal{O}(l_1+l_2,0)\otimes p_2^*\mathcal{O}(l_1,l_2)$.
\end{enumerate}
\end{prop}

One of the difficulties of the proof is that $\varphi_k$ becomes an immersion only after the previous strata have been blown up.
The following lemma describes
the situation. In this lemma, $\mathcal{E}^{d_2}_{d_1}$ denotes the vector bundle on $\bar{H}_{d_1}$
whose fiber over $\langle F\rangle$ is
$S_{d_2}/\langle F\rangle$, and hence whose projectivization is $\bar{H}_{d_1,d_2}$.
Moreover, when $h\leq k$, we will consider the following commutative diagram, in which $\mu$ denotes multiplication:

\begin{equation}\label{diagmultmult}
\begin{split}
\xymatrix{
\bar{H}_{d_1-k}\times \bar{H}_{k,d_2-d_1+k}\ar[r]^{\hspace{1.3em}\varphi_k}
&  \bar{H}_{d_1,d_2}    \\
\bar{H}_{d_1-k}\times\bar{H}_{k-h}\times\bar{H}_{h,d_2-d_1+h}\ar[r]^{\hspace{1.3em}(\mu,\Id)}\ar[u]_{(\Id,\varphi_h)} 
&\bar{H}_{d_1-h}\times\bar{H}_{h,d_2-d_1+h} \ar[u]_{\varphi_h}
}
\end{split}
\end{equation}

\begin{lem}\label{immdef} Let $1\leq k\leq d_1-1$.
\begin{enumerate} 
\item[(i)] The map $\varphi_k:\bar{H}_{d_1-k}\times\bar{H}_{k,d_2-d_1+k}\to\bar{H}_{d_1,d_2}$ is immersive on the open locus
$U=\{(P,[L,H])\in\bar{H}_{d_1-k}\times\bar{H}_{k,d_2-d_1+k}|\gcd(P,L,H)=1\}$.
\item[(ii)] Let $v\in\Ker ((d\varphi_k)_{(P,[L,H])})$ be non-zero, with $[L,H]\in W_h\setminus W_{h-1}$. Then $h<k$, and
$v$ is tangent to $\bar{H}_{d_1-k}\times W_h$.
\item[(iii)] The normal bundle to $\varphi_k|_U$ is $(p_1^*\mathcal{E}^{d_2+k}_{d_1-k}\otimes\mathcal{O}(1,1,1))|_U$.
\end{enumerate}
\end{lem}

\begin{proof}
We will prove analogous statements for the map $\tilde{\varphi}_k$.
It is easy to deduce the corresponding statements for $\varphi_k$ using
the rational map $\bar{H}_{d_1}\times\bar{H}_{d_2}\dashrightarrow\bar{H}_{d_1,d_2}$. 

Identifying $T_{\langle F\rangle}\bar{H}_d$ with $S_d/\langle F\rangle$, we see that:
\begin{alignat*}{2}
(d\tilde{\varphi}_k)_{(P,L,H)}:
S_{d_1-k}/\langle P\rangle\oplus & S_k/\langle L\rangle\oplus S_{d_2-d_1+k}/\langle H\rangle 
& \to  &  S_{d_1}/\langle PL\rangle\oplus S_{d_2}/\langle PH\rangle  \\
      &(A,B,C)&\mapsto  & (AL+BP, AH+CP).
\end{alignat*}
Let $\Pi=\gcd(P,L,H)$ be of degree $d$, and let $P',L',H'$ be such that $P=\Pi P'$, $L=\Pi L'$ and $H=\Pi H'$.
Using the formula above, it is straightforward to check that, if
$\Pi'\in S_{d}$, $(\Pi'P',-\Pi'L',-\Pi'H')\in\Ker(d\tilde{\varphi}_k)_{(P,L,H)}$.

On the other hand, if $(A,B,C)\in \Ker(d\tilde{\varphi}_k)_{(P,L,H)}$, we see that $PL|AL+BP$ and $PH|AH+CP$,
hence that $P'$ divides $AH'$, $AL'$ and of course $AP'$.
Consequently $P'|A$; this implies that $(A,B,C)$ is of the form described above.
In particular, if $\Pi=1$, $(d\tilde{\varphi}_k)_{(P,L,H)}$ is injective,
proving (i).

Let $\Gamma=\gcd(L',H')$, $L'=\Gamma L''$, $H'=\Gamma H''$, and $h=\deg(L'')$. By the above, if 
 $v\in\Ker(d\tilde{\varphi}_h)_{(P,L,H)}$ is non-zero, then $h<k$ and $v$ is of the form $(\Pi'P',-\Pi'L',-\Pi'H')$.
One sees that (ii) holds by checking that: $$v=d(\Id,\tilde{\varphi}_h)_{(Q,\Pi\Gamma,L'',H'')}(\Pi'P',-\Pi'\Gamma,0,0).$$

The Euler exact sequence realizes $T(\bar{H}_{d_1}\times\bar{H}_{d_2})$ as a natural quotient of
$S_{d_1}\otimes\mathcal{O}(1,0)\oplus S_{d_2}\otimes\mathcal{O}(0,1)$.
Restricting it to $\bar{H}_{d_1-k}\times\bar{H}_{k}\times\bar{H}_{d_2-d_1+k}$,
we identify $\tilde{\varphi}_k^*T(\bar{H}_{d_1}\times\bar{H}_{d_2})$ with a natural quotient of 
$S_{d_1}\otimes\mathcal{O}(1,1,0)\oplus S_{d_2}\otimes\mathcal{O}(1,0,1)$.
Now, if $(P,L,H)\in \bar{H}_{d_1-k}\times\bar{H}_{k}\times\bar{H}_{d_2-d_1+k}$, we have a linear map
$S_{d_1}\oplus S_{d_2}\to S_{d_2+k}$ given by
$(F,G)\mapsto LG-HF$, and these maps sheafify to induce a morphism of sheaves
$S_{d_1}\otimes\mathcal{O}(1,1,0)\oplus S_{d_2}\otimes\mathcal{O}(1,0,1)\to
S_{d_2+k}\otimes\mathcal{O}(1,1,1)$ on $ \bar{H}_{d_1-k}\times\bar{H}_{k}\times\bar{H}_{d_2-d_1+k}$.

Composing with the quotient map $S_{d_2+k}\otimes\mathcal{O}(1,1,1)\to p_1^*\mathcal{E}^{d_2+k}_{d_1-k}\otimes\mathcal{O}(1,1,1)$,
and noticing that, using the explicit description of $d\tilde{\varphi}_k$ above, the induced morphism factors through the normal
bundle $N_{\tilde{\varphi}_k}$, we obtain
a morphism of sheaves $\psi:N_{\tilde{\varphi}_k}\to p_1^*\mathcal{E}^{d_2+k}_{d_1-k}\otimes\mathcal{O}(1,1,1)$. 

To prove (iii), we need to check that $\psi$ is an isomorphism over
the locus where $\Pi=1$. Since it is a morphism between vector bundles of the same rank $d_1-k$,
it suffices to show $\psi$ induces a surjection at the level of fibers.
To do so, fix $(P,L,H)$ such that $\Pi=1$. The construction of $\psi$ shows that $\psi_{(P,L,H)}$ is induced by the linear map
$S_{d_1}\oplus S_{d_2}\to S_{d_2+k}/\langle Q\rangle$ given by
$(F,G)\mapsto LG-HF$. Thus, it suffices to prove that this map is surjective, i.e.
that every degree $d_2+k$ polynomial is a combination of $P,L,H$.

To do so, let $\Lambda=\gcd(L,H)$ be of degree $\lambda$,
and write $L=\Lambda L_1$ and $H=\Lambda H_1$. Notice that, since $\Pi=1$, $\gcd(\Lambda,P)=1$.
Thus, $(\Lambda,P)$ is a regular sequence on $\mathbb{P}^1$, giving rise to a Koszul exact sequence on $\mathbb{P}^1$:
$0\to\mathcal{O}(k-d_1-\lambda)\to\mathcal{O}(-\lambda)\oplus \mathcal{O}(k-d_1)\to\mathcal{O}\to 0$. Tensoring
by $\mathcal{O}(d_2+k)$, and taking global sections, one obtains a short exact sequence by vanishing of the appropriate $H^1$.
The surjectivity in this exact sequence shows precisely that
every degree $d_2+k$ polynomial may be written as a combination of $P$ and $\Lambda$.
It remains to express the coefficient of $\Lambda$, that is a degree
$d_2+k-\lambda$ polynomial, as a combination of $L_1$ and $H_1$. This is done in a similar fashion
 using the Koszul exact sequence associated to the regular sequence $(L_1,H_1)$ on $\mathbb{P}^1$.
\end{proof}

The proof of Proposition \ref{blowup} will proceed by induction, taking advantage of the inductive descriptions of the exceptional divisors.
Let us state the precise proposition that we will prove, whose
statement is adapted for an inductive proof, and from which Proposition \ref{blowup} will follow easily.

\begin{prop}\label{blowuprec}
\hspace{1em}
\begin{enumerate}
\item[(i)] There is a closed immersion $\varphi_k:\bar{H}_{d_1-k}\times \beta_1^{k-1}\bar{H}_{k,d_2-d_1+k}\to \beta_1^{k-1}\bar{H}_{d_1,d_2}$.
Its normal bundle is
$p_1^*\mathcal{E}^{d_2+k}_{d_1-k}(1)\otimes p_2^*(\mathcal{O}(1,1)(-E_1-\dots-E_{k-1}))$.
\item[(ii)] The variety $\beta_1^k\bar{H}_{d_1,d_2}$ is smooth as a blow-up of a smooth subvariety in a smooth variety.
Moreover, $E_k\simeq \bar{H}_{d_1-k,d_2+k}\times \hat{H}_{k,d_2-d_1+k}$.
\item[(iii)] If $j>k$, there is a cartesian diagram, in which $\mu$ denotes multiplication:
\begin{equation}\label{diagmultmap}
\xymatrix{
\mathcal{X}:=\bar{H}_{d_1-j}\times \beta_1^{k-1}\bar{H}_{j,d_2-d_1+j}\ar[r]^{\hspace{1.3em}\varphi_j}
&   \beta_1^{k-1}\bar{H}_{d_1,d_2}=:\mathcal{Z}    \\
\mathcal{W}:=\bar{H}_{d_1-j}\times\bar{H}_{j-k}\times\hat{H}_{k,d_2-d_1+k}\ar[r]^{\hspace{1.3em}(\mu,\Id)}\ar[u]_{(\Id,\varphi_k)} 
&\bar{H}_{d_1-k}\times\hat{H}_{k,d_2-d_1+k}=:\mathcal{Y}  \ar[u]_{\varphi_k}
}
\end{equation}
Moreover, if $z\in\mathcal{Z}$ is in the image of $\varphi_{k}$, $\varphi_j^{-1}(z)$ is finite of degree $\binom{d_1-k}{d_1-j}$.
\item[(iv)] If $k<j\leq d_1-1$ and $1\leq h\leq k$, there is a cartesian diagram:
\begin{equation}\label{diagdivexc}
\xymatrix{
\bar{H}_{d_1-j}\times \beta_1^{k-1}\bar{H}_{j,d_2-d_1+j}\ar[r]^{\hspace{2.7em}\varphi_j}
&   \beta_1^{k-1}\bar{H}_{d_1,d_2}    \\
\bar{H}_{d_1-j}\times \beta_1^{k}\bar{H}_{j,d_2-d_1+j}\ar[r]^{\hspace{2.7em}\varphi_j}\ar[u]_{\beta_k}
&   \beta_1^{k}\bar{H}_{d_1,d_2} \ar[u]_{\beta_k} \\
\bar{H}_{d_1-j}\times \beta_1^{k-h}\bar{H}_{j-h,d_2-d_1+j+h}\times \hat{H}_{h,d_2-d_1+h}
\ar[r]^{\hspace{2.7em}(\varphi_{j-h},\Id)}\ar[u]
& \beta_1^{k-h}\bar{H}_{d_1-h,d_2+h}\times\hat{H}_{h,d_2-d_1+h}\ar[u] \\
\bar{H}_{d_1-j}\times E_h\ar[r]\ar@{=}[u] &E_h\ar@{=}[u]
}
\end{equation}
\end{enumerate}
\end{prop}

\begin{proof}
We prove this proposition by induction on $d_1$, and when $d_1$ is fixed, by induction on $k$. When we use (i), (ii), (iii) or (iv), it is
always thanks to the induction hypothesis. We use without comment the fact that previously studied blow-ups are smooth blow-ups (ii).

The existence of the morphism $\varphi_k$ in (i) is given by \eqref{diagdivexc}.
Moreover, \eqref{diagdivexc} shows that if $1\leq h\leq k-1$, $\varphi_k^{-1}(E_h)=\bar{H}_{d_1-k}\times E_h$, and
the description of $\varphi_k: \varphi_k^{-1}(E_h)\to E_h$ given by \eqref{diagdivexc} implies, using (i), that it is a closed immersion.
On the other hand, it is easy to see that $\varphi_k$ is injective on the complement of these loci.
This shows that $\varphi_k$ is injective.

 Now suppose for contradiction that $\varphi_k$ is not immersive and
let $v$ be a non-zero tangent vector to $\bar{H}_{d_1-k}\times \beta_1^{k-1}\bar{H}_{k,d_2-d_1+k}$
at $x$ such that $d\varphi_k(v)=0$. For $0\leq l<k-1$, consider the cartesian diagrams deduced from \eqref{diagdivexc}:
\begin{equation}\label{diagblowup}
\xymatrix{
\bar{H}_{d_1-k}\times\beta_1^{l}\bar{H}_{k,d_2-d_1+k}\ar[r]^{\hspace{3em}\varphi_k}
&   \beta_1^{l}\bar{H}_{d_1,d_2}    \\
\bar{H}_{d_1-k}\times \beta_1^{k-1}\bar{H}_{k,d_2-d_1+k}\ar[r]^{\hspace{3em}\varphi_k}\ar[u]_{\beta_{l+1}^{k-1}}
&   \beta_1^{k}\bar{H}_{d_1,d_2} \ar[u]_{\beta_{l+1}^{k-1}} 
}
\end{equation}
Write $\beta_1^{k-1}(x)=(P,[L,H])$, and let $k-h=\deg(\gcd(L,H))$.
By Lemma \ref{immdef} (ii), $h<k$ and $d\beta_1^{k-1}(v)$ is tangent to $\bar{H}_{d_1-k}\times W_h$.
Consequently, $d\beta_{h+1}^{k-1}(v)$ is tangent to $\bar{H}_{d_1-k}\times E_h$. By \eqref{diagdivexc} and (i),
$\varphi_k|_{\bar{H}_{d_1-k}\times E_h}$ is immersive, and by
the commutativity of \eqref{diagblowup} for $l=h$, $d\varphi_k(d\beta_{h+1}^{k-1}(v))=0$, so that  $d\beta_{h+1}^{k-1}(v)=0$.
Now let $h<l\leq k$ be minimal such that $w=d\beta_{l+1}^{k-1}(v)\neq 0$.
Since $d\beta_{l}(w)=0$, $w$ is tangent to $\bar{H}_{d_1-k}\times E_l$. The same argument as above shows that $w=0$: a contradiction.

Since $\varphi_k$ is an immersion, its normal bundle is a vector bundle. By Lemma \ref{immdef} (iii)
and the behaviour of normal bundles under smooth blow-ups (\cite{Fulton} Appendix B 6.10),
it is isomorphic to $p_1^*\mathcal{E}^{d_2+k}_{d_1-k}(1)\otimes p_2^*(\mathcal{O}(1,1)(-E_1-\dots-E_{k-1}))$
on $(\beta_1^{k-1})^{-1}(U)$. Since this open set has complement of codimension $\geq 2$
and $\bar{H}_{d_1-k}\times \beta_1^{k-1}\bar{H}_{k,d_2-d_1+k}$ is smooth, hence normal, this isomorphism extends on all of 
$\bar{H}_{d_1-k}\times \beta_1^{k-1}\bar{H}_{k,d_2-d_1+k}$. This ends the proof of (i).

By (i), $\beta_1^k\bar{H}_{d_1,d_2}$ is the blow-up of a smooth subvariety in a smooth variety.
The computation of the normal bundle in (i) implies the required description of the exceptional divisor, proving (ii).

All maps in \eqref{diagmultmap} are well-defined by (iv).
Let us first prove the second assertion of (iii). When $z$ does not belong to an exceptional divisor, it is immediate
from the definition of $\varphi_{j}$. When $z$ belongs to an exceptional divisor $E_h$, it follows by induction of the descriptions
of $\varphi_j:\varphi_j^{-1}(E_h)\to E_h$ and $\varphi_k:\varphi_k^{-1}(E_h)\to E_h$ given in \eqref{diagdivexc}.

 The diagram \eqref{diagmultmap} is commutative because it is induced by the commutative diagram
\eqref{diagmultmult}. To show that it is cartesian, let us introduce the fiber product
$\mathcal{V}=\mathcal{X}\times_{\mathcal{Z}}\mathcal{Y}$ and the natural map $\mathcal{W}\to\mathcal{V}$. By (i),
$(\Id,\varphi_k): \mathcal{W}\to \mathcal{X}$ and $\varphi_k:\mathcal{Y}\to\mathcal{Z}$ hence also $\mathcal{V}\to\mathcal{X}$
are closed immersions. It follows that $\mathcal{W}\to\mathcal{V}$ is a closed immersion. Let $\mathcal{I}$ be its sheaf of ideals;
we want to show that $\mathcal{I}=0$.
By (iv), $\mathcal{X}\to\mathcal{Z}$ is a base-change of $\varphi_j:\bar{H}_{d_1-j}\times\bar{H}_{j,d_2-d_1+j}\to\bar{H}_{d_1,d_2}$,
hence it is finite. It follows that $\mathcal{V}\to\mathcal{Y}$ is finite.
Hence we may view $0\to\mathcal{I}\to\mathcal{O}_{\mathcal{V}}\to\mathcal{O}_{\mathcal{W}}\to 0$ as a short exact sequence of
coherent sheaves on $\mathcal{Y}$. Since $\mathcal{W}\to\mathcal{Y}$ is easily seen to be finite flat of degree $\binom{d_1-k}{d_1-j}$,
we get a short exact sequence $0\to\mathcal{I}_y\to(\mathcal{O}_{\mathcal{V}})_y\to(\mathcal{O}_{\mathcal{W}})_y\to 0$ for every $y\in Y$.
But both $(\mathcal{O}_{\mathcal{V}})_y$ and $(\mathcal{O}_{\mathcal{W}})_y$ are of dimension $\binom{d_1-k}{d_1-j}$, so that
$\mathcal{I}_y=0$. By Nakayama's lemma, $\mathcal{I}=0$, and (iii) holds.

The upper square of \eqref{diagdivexc} is cartesian because \eqref{diagmultmap} is.
When $h=k$, the lower square is the cartesian diagram relating the exceptional divisors of the blow-ups. The morphism between
those exceptional divisors is induced by the natural map between the normal bundles of the blown-up loci.
The explicit identification of these normal bundles made in Lemma \ref{immdef} (iii) allow to check that
this morphism is $\varphi_{j-k}$.
When $h<k$, the lower square is obtained by restricting the blow-up of $W_k$ to $E_h$. Its description follows of the description
of $\varphi_k:\varphi_k^{-1}(E_h)\to E_h$ given by \eqref{diagdivexc}, ending the proof of (iv).
\end{proof}

It is easy to deduce Proposition \ref{blowup}:

\begin{proof}[Proof of Proposition \ref{blowup}]
The variety $\hat{H}_{d_1,d_2}$ is smooth by Proposition \ref{blowuprec} (ii). The isomorphism
$E_k\simeq \hat{H}_{d_1-k,d_2+k}\times \hat{H}_{k,d_2-d_1+k}$ is provided by \eqref{diagdivexc}.

The computation of $E_j|_{E_k}$ when $j>k$ follows from the description of $\varphi_j:\varphi_j^{-1}(E_k)\to E_k$ given by \eqref{diagdivexc}.
The computation of $E_k|_{E_j}$ when $j>k$ follows from the description of $E_k|_{W_j}$ also given by \eqref{diagdivexc}.
The computation of $\mathcal{O}(E_k)|_{E_k}$ is a consequence of the description of the normal bundle to $\varphi_k$
in Proposition \ref{blowuprec} (i).

It is then easily seen by induction on $d_1$ that $(E_k)_{1\leq k\leq d_1-1}$ is a strict normal crossing divisor on 
$\hat{H}_{d_1,d_2}$. Indeed, for every $k$, $E_{k}$ is smooth and
$(E_j|_{E_{k}})_{j\neq k}$ is a strict normal crossing divisor on $E_k$ by induction. This implies 
that $(E_k)_{1\leq k\leq d_1-1}$ is a strict normal crossing divisor on 
$\hat{H}_{d_1,d_2}$.

Finally, the explicit expression of $\varphi_k$ shows that $\varphi_k^*\mathcal{O}(l_1,l_2)=\mathcal{O}(l_1+l_2,l_1,l_2)$
on $\bar{H}_{d_1-k}\times\bar{H}_{k,d_1-d_2+k}$. The computation of $\mathcal{O}(l_1,l_2)|_{E_k}$ follows.
\end{proof}

\begin{rem}\label{blowuptriv}
 It follows from Proposition \ref{blowuprec} (ii) that the last blow-up $\beta_{d_1-1}$ was the blow-up of a smooth
divisor, hence was not useful to construct $\hat{H}_{d_1,d_2}$. However, it was important to describe its
exceptional divisor $E_{d_1-1}$, that is the strict transform of the discriminant.
\end{rem}

\begin{rem}\label{identbir}
As it will be useful later, let us make explicit the identification of $E_k$ obtained above,
at least birationally.
It follows from the proof of Lemma \ref{immdef} (iii) and Proposition \ref{blowuprec} (ii) that the exceptional divisor
$E_k$ was birationally identified with $\bar{H}_{d_1-k,d_2+k}\times\bar{H}_{k,d_2-d_1+k}$
by sending a tangent vector induced by $[F,G]$
at a point $[PL,PH]\in W_k\subset \bar{H}_{d_1,d_2}$ to $([P,LG-HF],[L,H])$.
\end{rem}

\subsection{Linear systems on $\bar{H}_{d_1,d_2}$}

In this paragraph, we will construct several linear systems on $\bar{H}_{d_1,d_2}$,
generalizing the construction in Proposition \ref{cont1} of the linear system inducing the first contraction.

We fix a coordinate system $X_0,X_1$ on $\mathbb{P}^1$.
We will need to work with formal identities involving coefficients of polynomials of degrees $d_1$ and $d_2$.
For this reason we denote by $\mathfrak{M}_d$ the
set of monomials in $X_0,X_1$ of degree $d$, we let $(f^{(M)})_{M\in\mathfrak{M}_{d_1}}$
and $(g^{(M)})_{M\in\mathfrak{M}_{d_2}}$ be indeterminates, and we will
work in the ring $A=\mathbbm{k}[X_s,f^{(M)}, g^{(M)}]$ trigraded by the total degree in the $X_s$,
the $f^{(M)}$ and the $g^{(M)}$. Let $f=\sum_{M\in\mathfrak{M}_{d_1}} f^{(M)} M$ and $g=\sum_{M\in\mathfrak{M}_{d_2}} g^{(M)} M$.
We will often view elements of $A$ as polynomials in $X_0,X_1$ with coefficients
in $\mathbbm{k}[f^{(M)}, g^{(M)}]$. If $a\in A$, and $M=X_0^{d-j}X_1^j\in\mathfrak{M}_d$, we will
denote by $a^{(M)}=a^{(j)}$ the coefficient of $M$ in $a$.
If $a\in A$ and $(\lambda,\mu)\in\mathbbm{k}^2$, $a(\lambda,\mu)\in \mathbbm{k}[f^{(M)}, g^{(M)}]$
is obtained by evaluating $(X_0,X_1)$ at $(\lambda,\mu)$. Finally, if $(\lambda,\mu)\in\mathbbm{k}^2$,
$L_{\lambda,\mu}:=\lambda X_1-\mu X_0$ is a linear form vanishing on $(\lambda,\mu)$.

The following proposition is a variant of Euclid's algorithm formally applied to $g$ and $f$ (see Remark \ref{twisted}).

\begin{prop}\label{identities}
For all $0\leq u\leq d_1-1$, fix $(\lambda_u,\mu_u)\neq(0,0)\in \mathbbm{k}^2$.
Then there exist homogeneous elements $r_i, q_i\in A$ for $0\leq i\leq d_1-1$,
that depend algebraically on $\lambda_u,\mu_u$, such that $r_i$
is homogeneous of degree $(d_1-1-i,d_2-d_1+1+i,1+i)$ 
and such that the following identities hold in $A$:
\begin{alignat}{3}
f(\lambda_0,\mu_0)^{d_2-d_1+1}g &= fq_0+L_{\lambda_0,\mu_0}^{d_2-d_1+1}r_0 \label{identity1} \\
r_0(\lambda_1,\mu_1)^2f &= r_0q_1+f(\lambda_0,\mu_0)^{d_2-d_1+1}L_{\lambda_1,\mu_1}^2 r_1 \label{identity2} \\
r_{i-1}(\lambda_i,\mu_i)^2r_{i-2}
&= r_{i-1}q_i+r_{i-2}(\lambda_{i-1},\mu_{i-1})^2 L_{\lambda_i,\mu_i}^2 r_i, \label{identityi}
\end{alignat}
where $2\leq i\leq d_1-1$.
\end{prop}

In the course of the proof of this proposition, we will need the following lemma, that we prove first.

\begin{lem}\label{irredcoeff}
Let $0\leq j\leq d_1-1$. Suppose that $r_i, q_i\in A$ as in
Proposition \ref{identities} have been constructed for $0\leq i\leq j$. Then:
\begin{enumerate} 
\item [(i)] If $(\lambda,\mu)\neq(0,0)\in \mathbbm{k}^2$, $r_j(\lambda,\mu)$ does not vanish identically on $\tilde{W}_{j+1}$.
\item [(ii)] The coefficients $r_j^{(M)}$ of $r_j$ vanish on $\tilde{W}_{j}$.
\item [(iii)] For every $(\lambda,\mu)\neq(0,0)\in \mathbbm{k}^2$, $r_j(\lambda,\mu)$ is irreducible.
\end{enumerate}
\end{lem}

\begin{proof}
We use induction on $j$.

Let us first prove (i). By induction, the $r_i(\lambda_{i+1},\mu_{i+1})$ for $i<j$ do not vanish identically on $\tilde{W}_{j+1}$.
It is also clear that $f(\lambda_0,\mu_0)$ does not vanish identically on $\tilde{W}_{j+1}$.
Choose $(F,G)\in \tilde{W}_{j+1}$ general, so that it satisfies the three following conditions:
neither $f(\lambda_0,\mu_0)$ nor some $r_i(\lambda_{i+1},\mu_{i+1})$ vanish on it, $\Pi=\gcd(F,G)$ is of degree
$d_1-1-j$, and $\Pi$ does not vanish on $(\lambda,\mu)$ nor on some $(\lambda_i,\mu_i)$.
Substitute the coefficients of $F$ and $G$ in the $f^{(M)}$ and the $g^{(M)}$ in \eqref{identity1}, \eqref{identity2}
and \eqref{identityi} for $i\leq j$ to obtain polynomials $Q_i,R_i\in \mathbbm{k}[X_0,X_1]$,
and identities relating $G,F,Q_i,R_i$. 
These identities immediately show that $\Pi|R_j$. Suppose for contradiction that $R_j=0$
and let $-1\leq i< j$ be maximal such that $R_i\neq 0$
(where, by convention, $R_{-1}=F$). Then these same identities show that $R_i|\Pi$, which is impossible for degree reasons. 
Hence $R_j\neq 0$ and $\Pi=R_j$ for degree reasons. It follows that $R_j(\lambda,\mu)\neq 0$, as wanted.

 We argue in the same way to prove (ii): choose  $(F,G)\in \tilde{W}_{j}$ general, so that
neither $f(\lambda_0,\mu_0)$ nor some $r_i(\lambda_{i+1},\mu_{i+1})$ vanish on it (applying (i) by induction).
Substitute the coefficients of $F$ and $G$ in the $f^{(M)}$ and the $g^{(M)}$ in \eqref{identity1}, \eqref{identity2}
and \eqref{identityi} for $i\leq j$, to obtain polynomials $Q_i,R_i\in \mathbbm{k}[X_0,X_1]$ as before.
The identities relating $G,F,Q_i,R_i$ immediately show that $\gcd(F,G)|R_j$ which implies
$R_j=0$ for degree reasons. As a consequence, all the $r_j^{(M)}$ vanish on $(F,G)$ as wanted.

Let us finally check (iii). Let $h$ be an irreducible factor of $r_j(\lambda,\mu)$ vanishing on $\tilde{W}_{j}$ (using (ii)),
and let $(l_1,l_2)$ be its homogeneous degrees. The pull-back $\tilde{\varphi}_{j+1}^*h$ is a section
of $\mathcal{O}(l_1+l_2,l_1,l_2)$ on $\bar{H}_{d_1-j-1}\times\bar{H}_{j+1}\times\bar{H}_{d_2-d_1+j+1}$ vanishing on
$\bar{H}_{d_1-j-1}\times\Delta$, that is non-zero by (i).
Restricting it to a general fiber of the projection to the first factor, we get a non-zero section
of $\mathcal{O}(l_1,l_2)$ on $\bar{H}_{j+1}\times\bar{H}_{d_2-d_1+j+1}$ vanishing on $\Delta$.
But $\mathcal{O}(\Delta)=\mathcal{O}(d_2-d_1+j+1,j+1)$. Thus we necessarily have $l_1\geq d_2-d_1+j+1$ and
$l_2\geq j+1$, hence $h=r_j(\lambda,\mu)$.
\end{proof}

\begin{proof}[Proof of Proposition \ref{identities}]
To construct $q_0$ and $r_0$, we follow \cite{Oolqp} Lemme 2.6: we show by induction on $0\leq j\leq d_2-d_1+1$ that there exist
$q_{0,j},r_{0,j}\in A$ homogeneous of degrees $(d_2-d_1,j-1,1)$ and $(d_2-j,j,1)$ satisfying:
\begin{equation*}\label{divetapes}
f(\lambda_0,\mu_0)^{j}g = fq_{0,j}+L_{\lambda_0,\mu_0}^{j}r_{0,j}.
\end{equation*}
If $j=0$, take $q_{0,0}=0$ and $r_{0,0}=g$. If $q_{0,j},r_{0,j}$ have already been constructed,
fix a linear combination $L$ of $X_0$ and $X_1$ such that $L(\lambda_0,\mu_0)=1$ and set:
\begin{alignat*}{2}
q_{0,j+1}&=f(\lambda_0,\mu_0)q_{0,j}-r_{0,j}(\lambda_0,\mu_0)L^{d_2-d_1-j}L_{\lambda_0,\mu_0}^{j}\\
r_{0,j+1}&=(f(\lambda_0,\mu_0)r_{0,j}-r_{0,j}(\lambda_0,\mu_0)L^{d_2-d_1-j}f)/L_{\lambda_0,\mu_0}.
\end{alignat*}
Setting $q_0=q_{0,d_2-d_1+1}$ and $r_0=r_{0,d_2-d_1+1}$, we obtain the first identity \eqref{identity1}.

Now, treat temporarily the coefficients $r_0^{(M)}$ of $r_0$ as indeterminates. Applying the identity
\eqref{identity1} constructed above to $f$, $r_0$ and $(\lambda_1,\mu_1)$ (instead of $g$, $f$ and $(\lambda_0,\mu_0)$),
we obtain a formula of the form:
\begin{equation}\label{identity2tilde}
r_0(\lambda_1,\mu_1)^2 f = r_0q_1+L_{\lambda_1,\mu_1}^2\tilde{r}_1
\end{equation}
in $\mathbbm{k}[X_s,f^{(M)},r_0^{(M)}]$. Substituting in the indeterminate $r_0^{(M)}$
its value as an element of $\mathbbm{k}[f^{(M)}, g^{(M)}]$,
\eqref{identity2tilde} becomes an identity in $A$. To get \eqref{identity2}, it remains to check that
$f(\lambda_0,\mu_0)^{d_2-d_1+1}|\tilde{r}_1$. By specialization, and because
$\tilde{r}_1$ depends algebraically on $\lambda_0,\mu_0,\lambda_1,\mu_1$ by construction,
it suffices to check that under the additional hypothesis that
$[\lambda_0,\mu_0]\neq [\lambda_1,\mu_1]$. When this is the case, up to changing coordinates in $\mathbb{P}^1$,
it is possible to suppose that $(\lambda_0,\mu_0)=(1,0)$ and $(\lambda_1,\mu_1)=(0,1)$; we assume this is the case until
we finish to check that $f(\lambda_0,\mu_0)^{d_2-d_1+1}|\tilde{r}_1$.
Combining \eqref{identity1}
and \eqref{identity2tilde} to eliminate $r_0$,
one obtains an identity of the form $fa=X_0^2X_1^{d_2-d_1+1}\tilde{r}_1+(f^{(0)})^{d_2-d_1+1}b$, where $a,b\in A$
are homogeneous, $a$ being of degree $d_2-d_1+1$ in the $X_s$. 
As a consequence, for $0\leq j\leq d_2-d_1+2$, or for $j\in\{d_2,d_2+1\}$,
\begin{equation}\label{congru}\tag{$E_j$}
(f^{(0)})^{d_2-d_1+1}|(fa)^{(j)}.
\end{equation}
By $(\Eq_{d_2+1})$, $(f^{(0)})^{d_2-d_1+1}|a^{(d_2-d_1+1)}$.
Then, by $(\Eq_{d_2})$, $(f^{(0)})^{d_2-d_1+1}|a^{(d_2-d_1)}$.
Now suppose that $(f^{(0)})^{d_2-d_1+1}\nmid a$ and let $0\leq j\leq d_2-d_1-1$
be minimal such that $(f^{(0)})^{d_2-d_1+1}\nmid a^{(j)}$.
Considering equations $(\Eq_{j+k})$ for $0\leq k\leq d_2-d_1-j-1$,
we prove successively that $(f^{(0)})^{d_2-d_1-k}|a^{(j+k)}$ for $0\leq k\leq d_2-d_1-j-1$.
Then, considering equations $(\Eq_{j+k+1})$ for $d_2-d_1-j-1\geq k\geq 0$,
we prove successively that $(f^{(0)})^{d_2-d_1-k+1}|a^{(j+k)}$ for $d_2-d_1-j-1\geq k\geq 0$.
In particular, $(f^{(0)})^{d_2-d_1+1}|a^{(j)}$, which is a contradiction.
We have proved that $(f^{(0)})^{d_2-d_1+1}|a$, hence that $(f^{(0)})^{d_2-d_1+1}|\tilde{r}_1$.
It follows that \eqref{identity2tilde} gives rise to an identity of the required form \eqref{identity2}.

Let $2\leq i\leq d_1-1$, suppose that $r_j$ and $q_j$
have been constructed for all $j<i$, and let us construct $r_i$ and $q_i$.
Treat temporarily
the coefficients of $r_{i-3}$ and $r_{i-2}$ as indeterminates (where, by convention, $r_{-1}=f$). Applying the identities
\eqref{identity1} and \eqref{identity2} constructed above to $r_{i-3}$, $r_{i-2}$, $(\lambda_{i-1},\mu_{i-1})$
and $(\lambda_i,\mu_i)$ (instead of $g$, $f$, $(\lambda_0,\mu_0)$
and $(\lambda_1,\mu_1)$), we obtain formulas of the form: 
\begin{alignat}{2}
r_{i-2}(\lambda_{i-1},\mu_{i-1})^{2}r_{i-3} 
= r_{i-2}\tilde{q}_{i-1}+L_{\lambda_{i-1},\mu_{i-1}}^{2}\tilde{r}_{i-1} \label{identityi1} \\
\tilde{r}_{i-1}(\lambda_i,\mu_i)^2r_{i-2} 
= \tilde{r}_{i-1}\tilde{q}_{i}+r_{i-2}(\lambda_{i-1},\mu_{i-1})^{2} L_{\lambda_i,\mu_i}^2 \tilde{r}_{i} \label{identityi2} 
\end{alignat}
in $\mathbbm{k}[X_s,r_{i-3}^{(M)},r_{i-2}^{(M)}]$. Substituting in $r_{i-3}^{(M)}$ and $r_{i-2}^{(M)}$
their values as elements of $\mathbbm{k}[f^{(M)}, g^{(M)}]$, 
\eqref{identityi1} and \eqref{identityi2} become identities in $A$.
Since \eqref{identityi1} and \eqref{identityi} for $i-1$ have been constructed
exactly in the same way, it follows that $\tilde{q}_{i-1}=q_{i-1}$ and that
$\tilde{r}_{i-1}=r_{i-3}(\lambda_{i-2},\mu_{i-2})^2 r_{i-1}$.
Since, by construction, the coefficients of $\tilde{q}_{i}$
are polynomials of bidegree $(1,1)$ in the coefficients of $r_{i-2}$ and $\tilde{r}_{i-1}$,
we see that $r_{i-3}(\lambda_{i-2},\mu_{i-2})^2 |\tilde{q}_{i}$.
Write $\tilde{q}_{i}=r_{i-3}(\lambda_{i-2},\mu_{i-2})^2q_{i}$. Equation \eqref{identityi2} becomes:
\begin{alignat*}{2}
r_{i-3}(\lambda_{i-2},\mu_{i-2})^4 (r_{i-1}(\lambda_i,\mu_i)^2r_{i-2}-
 r_{i-1}q_{i})=r_{i-2}(\lambda_{i-1},\mu_{i-1})^{2}
L_{\lambda_i,\mu_i}^2 \tilde{r}_{i}.
\end{alignat*}
By Lemma \ref{irredcoeff} (iii), $r_{i-3}(\lambda_{i-2},\mu_{i-2})$ and $r_{i-2}(\lambda_{i-1},\mu_{i-1})$
are irreducible. Since their degrees are different, they are prime to each other,
so that $r_{i-3}(\lambda_{i-2},\mu_{i-2})^4| \tilde{r}_{i}$. Dividing by 
$r_{i-3}(\lambda_{i-2},\mu_{i-2})^4$ leads to an identity of the required form \eqref{identityi}.
\end{proof}

\begin{prop}
Keep the notations of Proposition \ref{identities}. Let $0\leq i\leq d_1-1$ and $M\in\mathfrak{M}_{d_1-1-i}$. 
Then $r_i^{(M)}\in H^0(\bar{H}_{d_1}\times\bar{H}_{d_2},\mathcal{O}(d_2-d_1+1+i,1+i))$
comes from a section in $H^0(\bar{H}_{d_1,d_2},\mathcal{O}(d_2-d_1+1+i,1+i))$
via the rational map $\bar{H}_{d_1}\times\bar{H}_{d_2}\dashrightarrow\bar{H}_{d_1,d_2}$.
\end{prop}

\begin{proof}
Since, by construction, the $r_{i}^{(M)}$ for $i>0$ are rational
functions in the $r_{0}^{(M)}$ and the $f^{(M)}$, it suffices to treat the case $i=0$.

But this case has already been dealt with in the proof of Proposition \ref{cont1}.
\end{proof}

We denote by $\Lambda_i$ be the linear system on $\bar{H}_{d_1,d_2}$ generated by the $r^{(i)}_M$
for all possible choices of scalars $\lambda_0,\dots,\lambda_{d_1-1},\mu_0,\dots,\mu_{d_1-1}$ and of monomials $M$.

\begin{prop}\label{baselocus}
If $0\leq i\leq d_1-1$, the base locus of $\Lambda_i$ is $W_i$.
\end{prop}

\begin{proof}
Lemma \ref{irredcoeff} (ii) shows that $W_i$ is included in the base locus of $\Lambda_i$. It remains to show the other inclusion.
Suppose that $(F,G)\notin \tilde{W}_i$. Choose $(\lambda_0,\mu_0)$ such that $F(\lambda_0,\mu_0)\neq 0$.  By substituting the coefficients
of $F$ and $G$ in the $f^{(M)}$ and the $g^{(M)}$ in \eqref{identity1}, one obtains a polynomial $R_0$. By choice of $(\lambda_0,\mu_0)$,
and since $(F,G)\notin \tilde{W}_0$, $R_0\neq 0$. Choose $(\lambda_1,\mu_1)$ such that
$R_0(\lambda_1,\mu_1)\neq 0$, and use \eqref{identity2} to construct a polynomial $R_1$. Iterating this process, one eventually chooses
$\lambda_0,\dots,\lambda_i, \mu_0,\dots,\mu_i$ inducing a non-zero $R_i$. If $M$ is a monomial with non-zero coefficient in $R_i$, $r_i^{(M)}$
induces a section in $\Lambda_i$ that does not vanish on $[F,G]$, as wanted.
\end{proof}

\begin{rem} \label{twisted}
When $(\lambda_u,\mu_u)=(1,0)$ for all $u$, the identities provided by Proposition \ref{identities}
really are Euclid's algorithm formally applied to $g$ and $f$ (with quotients $q_i$ and remainders $r_i$).
It was however necessary for our purposes to authorize variants of this algorithm, that is to
allow $(\lambda_u,\mu_u)$ to depend on $u$.

Indeed, if we had insisted that $(\lambda_u,\mu_u)$ did not depend on $u$, the linear systems
$\Lambda'_i$ we would have constructed would have been too small
in finite charac\-teristic. For instance, it may be checked that,
in characteristic $p\geq 3$, $[X_0^p,X_1^{2p}]$ would have been in the base locus of $\Lambda'_1$, so that
Proposition \ref{baselocus}  would not have held.
\end{rem}

\subsection{Base-point freeness}
In this paragraph, we will show that $\Lambda_i$ induces a base-point free linear system $\hat{\Lambda}_i$ on $\hat{H}_{d_1,d_2}$.
Since, by Proposition \ref{baselocus}, the base locus of $\Lambda_i$ is $W_i$, we will need to study $\Lambda_i$
around the $W_k$ for $1\leq k\leq i$. 

For this purpose, we introduce
homogeneous polynomials $p,l,h,\phi$ and $\gamma$ in $X_0,X_1$ of respective degrees $d_1-k,k, d_2-d_1+k,d_1$ and $d_2$,
with indeterminate coefficients $p^{(M)},l^{(M)},h^{(M)},\phi^{(M)}$ and $\gamma^{(M)}$. We define $f=pl+t\phi$ and
$g=ph+t\gamma$, and we let $f^{(M)}$ and $g^{(M)}$ be their coefficients.
Substituting those values in the indeterminates $f^{(M)}$ and $g^{(M)}$ in the identities
\eqref{identity1}, \eqref{identity2} and \eqref{identityi}, we get identities in
$B[t]=\mathbbm{k}[X_s,p^{(M)},l^{(M)},h^{(M)},\phi^{(M)},\gamma^{(M)}][t]$. In all this paragraph,
$r_i,q_i$ will be viewed in this way as elements of $B[t]$. By convention, we define $r_{-1}:=f$ and $r_{-2}:=g$.

Studying these identities at the lowest order in $t$
will give informations about $\Lambda_i$ at the point $[pl,ph]\in W_k$ in the
tangent direction $[\phi,\gamma]$.
If $b\in B[t]$, we will write $b=\sum_lb^{[l]}t^l$ with $b^{[l]}\in B$:
$b^{[l]}$ is the order $l$ term of $b$.

The main idea is that, when applying Proposition \ref{identities} to $g=ph+t\gamma$ and $f=pl+t\gamma$,
the $k$ first remainders are related to the remainders of
Proposition \ref{identities} applied to $h$ and $l$ (Lemma \ref{dividiv}),
and the $d_1-k$ last remainders are related to the remainders
of Proposition \ref{identities} applied to $\gamma l-h\phi$ and $p$
(Lemma \ref{dividiv2}).

\vspace{1em}

If $0\leq i\leq k-1$, we define $\bar{r}_i$ and $\bar{q}_i$ to be the remainders
and the quotients obtained by applying Proposition \ref{identities}
to $h$ and $l$ (instead of to $g$ and $f$) using the scalars
$\lambda_0,\dots,\lambda_{k},\mu_0,\dots,\mu_{k}$.

\begin{lem}\label{dividiv}
For $0\leq i\leq k-1$, $r_i^{[0]}=p(\lambda_0,\mu_0)^{d_2-d_1+1}\prod_{j=1}^ip(\lambda_j,\mu_j)^2p\bar{r}_i$.
\end{lem}

\begin{proof}
Notice that the identities \eqref{identity1}, \eqref{identity2} and \eqref{identityi} for $h$ and $l$ (resp. for $ph$ and $pl$)
are obtained by applying the very same algorithm.
It follows that it is possible to identify term by term these two sets of identities.
For instance, \eqref{identity1} for $h$ and $l$ (resp. for $ph$ and $pl$) read:
\begin{alignat*}{2}
l(\lambda_0,\mu_0)^{d_2-d_1+1}h &= l\bar{q}_0+L_{\lambda_0,\mu_0}^{d_2-d_1+1}\bar{r}_0\\
(pl)(\lambda_0,\mu_0)^{d_2-d_1+1}ph &= plq_0^{[0]}+L_{\lambda_0,\mu_0}^{d_2-d_1+1}r_0^{[0]}.
\end{alignat*}
Identifying term by term these two identities, we get
$r_0^{[0]}=p(\lambda_0,\mu_0)^{d_2-d_1+1}p\bar{r}_0$ and
$q_0^{[0]}=p(\lambda_0,\mu_0)^{d_2-d_1+1}\bar{q}_0$, which is the $i=0$ case of what we want.
To prove the general case, we argue by induction on $i$  and compare successively
identities \eqref{identity2} and \eqref{identityi} for $2\leq i\leq k-1$ applied to
$h$ and $l$ (resp. to $ph$ and $pl$).
\end{proof}

In particular, for $0\leq i\leq k-1$, $p|r_i^{[0]}$. We define $s_i$ to be such that $r_i^{[0]}=ps_i$.

If $0\leq i\leq d_1-k-1$,
we define $\tilde{r}_i$ and $\tilde{q}_i$ to be the remainders and the
quotients obtained by applying Proposition \ref{identities}
to $(-1)^k(\gamma l-h\phi)$ and $p$ (instead of to $g$ and $f$) using the scalars
$\lambda_0,\lambda_{k+1},\dots,\lambda_{d_1-1},\mu_0,\mu_{k+1},\dots,\mu_{d_1-1}$.

\begin{lem}\label{dividiv2} Let $1\leq k\leq i\leq d_1-1$. Then:
\begin{enumerate}
\item[(i)] $t^{i-k+1}|r_i$.
\item[(ii)] If $\lambda_0=\dots=\lambda_{k}$ and $\mu_0=\dots=\mu_{k}$, $r_i^{[i-k+1]}=\bar{r}_{k-1}\tilde{r}_{i-k}$.
\end{enumerate}
\end{lem}

\begin{proof} Let us fix $k$, we will use induction on $i$ and start by proving the case $i=k$. Write down
\eqref{identityi} for $i=k$ at order $0$:
$$r_{k-1}^{[0]}(\lambda_k,\mu_k)^2r^{[0]}_{k-2}
=r^{[0]}_{k-1}q^{[0]}_k+r^{[0]}_{k-2}(\lambda_{k-1},\mu_{k-1})^2 L_{\lambda_k,\mu_k}^2 r^{[0]}_k.$$
By Lemma \ref{dividiv}, $p|r^{[0]}_{k-2}$ and  $p|r_{k-1}^{[0]}$. 
Since $r^{[0]}_{k-2}(\lambda_{k-1},\mu_{k-1})\neq 0$ (use Lemma \ref{dividiv} and Lemma \ref{irredcoeff} (i)),
it follows that $p|r^{[0]}_k$, hence that $r^{[0]}_k=0$ for degree reasons, proving (i).

By Lemma \ref{dividiv}, $\bar{r}_{k-1}|r_{k-1}^{[0]}$. Since, by construction, the coefficients of $q^{[0]}_k$
are polynomials of bidegree $(1,1)$ in the coefficients of $r_{k-1}^{[0]}$ and $r_{k-2}^{[0]}$,
$\bar{r}_{k-1}|q_{k}^{[0]}$. Now write down \eqref{identityi} for $i=k$ at order $1$: 
\begin{alignat*}{2}
 &r_{k-1}^{[0]}(\lambda_k,\mu_k)^2r^{[1]}_{k-2}+2r_{k-1}^{[0]}(\lambda_k,\mu_k)r_{k-1}^{[1]}(\lambda_k,\mu_k)r^{[0]}_{k-2} \\
=r^{[0]}_{k-1}q^{[1]}_k&+r^{[1]}_{k-1}q^{[0]}_k+r^{[0]}_{k-2}(\lambda_{k-1},\mu_{k-1})^2 L_{\lambda_k,\mu_k}^2 r^{[1]}_k.
\end{alignat*}

Since $\bar{r}_{k-1}|r_{k-1}^{[0]}$, $\bar{r}_{k-1}|q_{k}^{[0]}$ and $\bar{r}_{k-1}$ is prime to
$r^{[0]}_{k-2}(\lambda_{k-1},\mu_{k-1})$ (use Lemma \ref{dividiv} and notice that
$\bar{r}_{k-1}$ and $\bar{r}_{k-2}(\lambda_{k-1},\mu_{k-1})$ are prime to each other since they are irreducible
by Lemma \ref{irredcoeff} (i) but of different degrees), we see that $\bar{r}_{k-1}|r^{[1]}_k$.
We will write $r^{[1]}_k=\bar{r}_{k-1}\rho_0$.
 
Now assume that $\lambda_0=\dots=\lambda_{k}$ and $\mu_0=\dots=\mu_{k}$.
Let us prove by induction on $0\leq j\leq k$
that there exists $a_j\in B$ such that:
\begin{alignat*}{2}\label{divdivinterm}
(pl)(\lambda_0,\mu_0)^{d_2-d_1+1}(\gamma l-h\phi)
&=a_0 p+L_{\lambda_0,\mu_0}^{d_2-d_1+1}(lr_0^{[1]}-s_0\phi)\\
r_{j-1}^{[0]}(\lambda_0,\mu_0)^2(\gamma l-h\phi)
&=a_j p+(-1)^jL_{\lambda_0,\mu_0}^{d_2-d_1+2j+1}(s_{j-1}r_j^{[1]}-s_jr_{j-1}^{[1]})\textrm{ if }j>0.
\end{alignat*}
To prove the $j=0$ case, write down \eqref{identity1} at order $0$ and $1$ to get:
\begin{alignat*}{2}
(pl)(\lambda_0,\mu_0)^{d_2-d_1+1}h&=lq_0^{[0]}+L_{\lambda_0,\mu_0}^{d_2-d_1+1}s_0 \\
(pl)(\lambda_0,\mu_0)^{d_2-d_1+1}\gamma&=\phi q_0^{[0]}+L_{\lambda_0,\mu_0}^{d_2-d_1+2j+1}r^{[1]}_0 +a'_0p.
\end{alignat*}
Combining these two identities leads to an identity of the required form for $j=0$.
To obtain the required identity for $j$ from the one for $j-1$,
modify it using a suitable combination of the order
$0$ and $1$ terms of \eqref{identity2} if $j=1$ or \eqref{identityi} for $i=j$ if $j>1$.

When $j=k$, remember from above that $s_k=0$, use the expressions for $s_{j-1}$ and $r_{j-1}^{[0]}(\lambda_0,\mu_0)$ obtained
in Lemma \ref{dividiv}, and divide by an appropriate commom factor to obtain an expression of the form:
\begin{equation*}
p(\lambda_0,\mu_0)^{d_2-d_1+2k+1}(\gamma l-h\phi)=bp+(-1)^kL_{\lambda_0,\mu_0}^{d_2-d_1+2k+1}\rho_0.
\end{equation*}
On the other hand, we have:
\begin{equation*}
p(\lambda_0,\mu_0)^{d_2-d_1+2k+1}(-1)^k(\gamma l-h\phi)=\tilde{q}_0p+L_{\lambda_0,\mu_0}^{d_2-d_1+2k+1}\tilde{r}_0.
\end{equation*}
Combining these two equations,
we get $(b-(-1)^k\tilde{q}_0)p=(-1)^k(\tilde{r}_0-\rho_0)L_{\lambda_0,\mu_0}^{d_2-d_1+2k+1}$. For the left-hand side
to vanish at order $d_2-d_1+2k+1$ at $(\lambda_0,\mu_0)$, we need to have $b=\tilde{q}_0$, hence $\rho_0=\tilde{r}_0$.
This proves (ii) and finishes the $i=k$ case.

Suppose from now on that the statement holds for $i-1$ and let us check that it holds for $i$. Consider
equation \eqref{identityi}:
\begin{equation}\label{encore}
 r_{i-1}(\lambda_i,\mu_i)^2r_{i-2}=r_{i-1}q_i+r_{i-2}(\lambda_{i-1},\mu_{i-1})^2 L_{\lambda_i,\mu_i}^2 r_i.
\end{equation}
By induction, $t^{i-k}|r_{i-1}$ and $t^{i-k-1}|r_{i-2}$. Since, by construction, the coefficients of $q_i$
are polynomials of bidegree $(1,1)$ in the coefficients of $r_{i-1}$ and $r_{i-2}$, $t^{2i-2k-1}|q_i$.
Notice that, applying (ii) by induction and Lemma \ref{irredcoeff} (i),
$r^{[i-k-1]}_{i-2}(\lambda_{i-1},\mu_{i-1})\neq 0$ in the particular
case when $\lambda_0=\dots=\lambda_{k}$ and $\mu_0=\dots=\mu_{k}$.
Hence, $r^{[i-k-1]}_{i-2}(\lambda_{i-1},\mu_{i-1})\neq 0$ for general values of the $\lambda_u,\mu_u$. It follows
from \eqref{encore} that for general values of the $\lambda_u,\mu_u$, $t^{i-k+1}|r_i$. By specialization,
this holds for any values of the $\lambda_u,\mu_u$, proving (i).

To check (ii), assume that $\lambda_0=\dots=\lambda_{k}$ and $\mu_0=\dots=\mu_{k}$
and write \eqref{encore} at order $3i-3k-1$ in $t$:
\begin{equation}\label{toujours}
 r_{i-1}^{[i-k]}(\lambda_i,\mu_i)^2r_{i-2}^{[i-k-1]}
=r_{i-1}^{[i-k]}q_i^{[2i-2k-1]}+r_{i-2}^{[i-k-1]}(\lambda_{i-1},\mu_{i-1})^2 L_{\lambda_i,\mu_i}^2 r_i^{[i-k+1]}.
\end{equation}

By induction (and Lemma \ref{dividiv} if $i=k+1$),
$\bar{r}_{k-1}|r_{i-1}^{[i-k]}$ and $\bar{r}_{k-1}|r_{i-2}^{[i-k-1]}$.
Since, by construction, the coefficients of $q_i$
are polynomials of bidegree $(1,1)$ in the coefficients of $r_{i-1}$ and $r_{i-2}$, $\bar{r}_{k-1}^2|q_i$.
Using (ii) by induction (or Lemma \ref{dividiv} if $i=k+1$),
we see that $\bar{r}_{k-1}^2\nmid r_{i-2}^{[i-k-1]}(\lambda_{i-1},\mu_{i-1})^2$. It follows
that $\bar{r}_{k-1}|r_{i}^{[i-k+1]}$: let us write $r_{i}^{[i-k+1]}=\bar{r}_{k-1}\rho_{i-k}$.
Dividing \eqref{toujours} by $\bar{r}_{k-1}^3$, and also by
$p(\lambda_0,\mu_0)^{d_2-d_1+2k+1}$ if $i=k+1$, we get an identity of the form:
\begin{alignat*}{2}
 \tilde{r}_{0}(\lambda_{k+1},\mu_{k+1})^2p
=\tilde{r}_{0}q+p(\lambda_{0},\mu_{0})^{d_2-d_1+2k+1} L_{\lambda_{k+1},\mu_{k+1}}^2\rho_{1}\textrm{ if }i=k+1,\\
\tilde{r}_{i-k-1}(\lambda_i,\mu_i)^2\tilde{r}_{i-k-2}
=\tilde{r}_{i-k-1}q+\tilde{r}_{i-k-2}(\lambda_{i-1},\mu_{i-1})^2 L_{\lambda_i,\mu_i}^2\rho_{i-k}\textrm{ if }i>k+1.
\end{alignat*}
On the other hand, we have:
\begin{alignat*}{2}
 \tilde{r}_{0}(\lambda_{k+1},\mu_{k+1})^2p
=\tilde{r}_{0}\tilde{q}_{1}+p(\lambda_{0},\mu_{0})^{d_2-d_1+2k+1} L_{\lambda_{k+1},\mu_{k+1}}^2\tilde{r}_{1}\textrm{ if }i=k+1,\\
\tilde{r}_{i-k-1}(\lambda_i,\mu_i)^2\tilde{r}_{i-k-2}
=\tilde{r}_{i-k-1}\tilde{q}_{i-k}+\tilde{r}_{i-k-2}(\lambda_{i-1},\mu_{i-1})^2 L_{\lambda_i,\mu_i}^2\tilde{r}_{i-k}\textrm{ if }i>k+1.
\end{alignat*}
Substracting these two equations, and considering the order of vanishing of each term at $(\lambda_i,\mu_i)$,
it follows that $\rho_{i-k}=\tilde{r}_{i-k}$, proving (ii).
\end{proof}

\begin{prop}\label{multiplicity}
If $1\leq k\leq i\leq d_1-1$, $\mult_{W_k}(\Lambda_i)\geq i-k+1$.
\end{prop}

\begin{proof} 
By Lemma \ref{dividiv2} (i), if $i\geq k$, $t^{i-k+1} |r_i^{(M)}$ for every choice of scalars $\lambda_u,\mu_u$
and of monomials $M$. This means that $r_i^{(M)}$ vanishes at order $i-k+1$ on $W_k$, as wanted.
\end{proof}

Define $\mathcal{L}_i:=\mathcal{O}(d_2-d_1+1+i,1+i)(-\sum_{k=1}^{i}(i-k+1)E_k)$: it is a line bundle on $\hat{H}_{d_1,d_2}$.
By Proposition \ref{multiplicity},
$\hat{\Lambda}_i:=(\beta_1^{d_1-1})^*\Lambda_i-\sum_{k=1}^{i}(i-k+1)E_k$ is
a linear system included in $|\mathcal{L}_i|$. We will denote by $\hat{r}_i^{(M)}$ the section of $\mathcal{L}_i$
induced by $r_i^{(M)}$. By abuse of notation, we will also denote by $r_i^{(M)}$ (resp. $\hat{r}_i^{(M)}$) the divisors
in $\Lambda_i$ (resp. $\hat{\Lambda}_i$) they induce.

\begin{lem}\label{restrli}
Let $1\leq k\leq d_1-1$ and $0\leq i\leq d_1-1$.
\begin{enumerate}
\item[(i)] $\mathcal{L}_{d_1-1}$ is trivial.
\item[(ii)] If $k\leq i$, $\mathcal{L}_i|_{E_k}\simeq p_1^*\mathcal{L}_{i-k}$. 
\item[(iii)] If $k>i$, $\mathcal{L}_i|_{E_k}\simeq p_1^*\mathcal{O}(d_2-d_1+2i+2,0)\otimes p_2^*\mathcal{L}_i$.
\end{enumerate}
\end{lem}

\begin{proof} Let us first prove (i) by induction on $d_1$, the statement being trivial if $d_1=1$ (see Convention \ref{conv11}).
Take a divisor $D$ in $\Lambda_{d_1-1}$: by Proposition \ref{baselocus}, there exist some, and it necessarily contains the discriminant
$\Delta=W_{d_1-1}$. Since their degrees coincide, we have in fact
$D=\Delta$. This means that $\Lambda_{d_1-1}$ is reduced to a point: the
discriminant. As a consequence, $\hat{\Lambda}_{d_1-1}$ is reduced to a point,
that corresponds to a linear combination of $E_1,\dots, E_{d_1-2}$.
By Proposition \ref{blowup} (iii), (iv), (v) and (vi), we see that
 $\mathcal{L}_{d_1-1}|_{E_{d_1-1}}$ is isomorphic to $p_2^*\mathcal{L}_{d_1-2}$,
hence is trivial by the induction hypothesis. It follows that the linear combination
of $E_1,\dots, E_{d_1-2}$ inducing $\hat{\Lambda}_{d_1-1}$
is trivial, and that $\mathcal{L}_{d_1-1}=\mathcal{O}(\hat{\Lambda}_{d_1-1})$ is trivial.

Part (iii) is a formal consequence of 
Proposition \ref{blowup} (iii), (iv), (v) and (vi). Similarly,
$\mathcal{L}_i|_{E_k}\simeq p_1^*\mathcal{L}_{i-k}\otimes p_2^*\mathcal{L}_{k-1}$ for $k\leq i$.
Applying (i), we get (ii).
\end{proof}

\begin{prop}
If $0\leq i\leq d_1-1$, the linear system $\hat{\Lambda}_i$ is base-point free on $\hat{H}_{d_1,d_2}$.
\end{prop}

\begin{proof}
We use induction on $d_1$. By Proposition \ref{baselocus},
$\hat{\Lambda}_i$ has no base-point outside of
the $(E_k)_{1\leq k\leq i}$. Now, fix $1\leq k\leq i$. We are going to prove below that 
$p_1^*\hat{\Lambda}_{i-k}\subset\hat{\Lambda}_i|_{E_k}$ (see Lemma \ref{restrli} (ii)):
this will imply by induction that $\hat{\Lambda}_i$
has no base-point on $E_k$, and hence that it is base-point free. To do this, fix $\lambda_0',\dots,\lambda_{d_1-k-1}',
\mu_0',\dots,\mu_{d_1-k-1}'$. We need to show that the sections $p_1^*r_{i-k}^{(M)}$
associated to these $\lambda'_u,\mu'_u$ appear in $\hat{\Lambda}_i|_{E_k}$.
For this purpose, we define $\lambda_u=\lambda'_0$ and $\mu_u=\mu'_0$ if $0\leq u\leq k$ and
$\lambda_u=\lambda'_{u-k}$ and $\mu_u=\mu'_{u-k}$ if $k\leq u\leq d_1-1$.

Recall from Remark \ref{identbir} that the exceptional divisor
$E_k$ was birationally identified to $\bar{H}_{d_1-k,d_2+k}\times\bar{H}_{k,d_2-d_1+k}$
by sending a tangent vector induced by $[F,G]$
at a point $[PL,PH]\in W_k\subset \bar{H}_{d_1,d_2}$ to $([P,LG-HF],[L,H])$.
Hence, it follows from Lemma \ref{dividiv2} (ii)
that $\hat{r}_i^{(M)}$ induces on $\bar{H}_{d_1-k,d_2+k}\times\bar{H}_{k,d_2-d_1+k}$ the same divisor as
$p_1^*r_{i-k}^{(M)}\cdot p_2^*r_{k-1}^{(M)}$. Since, by the proof of Lemma \ref{restrli} (i), $r_{k-1}^{(M)}$
is the discriminant, it follows that $\hat{r}_i^{(M)}$ and $p_1^*\hat{r}_{i-k}^{(M)}$ coincide up to a combination of the exceptional divisors.
But since they are sections of the same line bundles by Lemma \ref{restrli} (ii),
this implies that they in fact coincide. Hence, $p_1^*r_{i-k}^{(M)}$ appears in $\hat{\Lambda}_i|_{E_k}$, as wanted.
\end{proof}

\subsection{The MMP for $\bar{H}_{d_1,d_2}$}

In the previous paragraph, we constructed several base-point free linear systems on $\hat{H}_{d_1,d_2}$.
Here, we describe the contractions they induce, and
use them to construct the MMP for $\bar{H}_{d_1,d_2}$.
We define $\mathcal{L}_{-1}:=\mathcal{O}(1,0)$.

\begin{prop}\label{nef}
The nef cone of $\hat{H}_{d_1,d_2}$ is simplicial, and generated by the semi-ample line bundles $(\mathcal{L}_i)_{-1\leq i\leq d_1-2}$.
\end{prop}

\begin{proof}
First, the case $d_1=1$ being trivial, let us suppose $d_1\geq 2$.

Since $\hat{H}_{d_1,d_2}$ has been constructed from $\bar{H}_{d_1,d_2}$ by blowing-up $d_1-2$
times an irreducible smooth subvariety of codimension $\geq2$ (see Remark \ref{blowuptriv}),
its Picard group has rank $d_1$ and is generated by $\mathcal{O}(1,0)$, $\mathcal{O}(0,1)$ and
the $(E_k)_{1\leq k\leq d_1-2}$. It follows that the
line bundles $(\mathcal{L}_i)_{-1\leq i\leq d_1-2}$ form a basis of the Picard group of
$\hat{H}_{d_1,d_2}$. Since they are all semi-ample, the cone they generate
is included in the nef cone. To prove the reverse inclusion, it suffices to construct effective
curves $(C_i)_{-1\leq i\leq d_1-2}$ in $\hat{H}_{d_1,d_2}$ such that $C_i\cdot\mathcal{L}_j$
is zero if and only if $i\neq j$. Indeed those curves will induce inequalities satisfied by the nef cone showing it is
included in the span of the $(\mathcal{L}_i)_{-1\leq i\leq d_1-2}$. 

Let us construct these curves by induction on $d_1$.
Since $E_1\simeq \hat{H}_{d_1-1,d_2+1}\times \hat{H}_{1,d_2-d_1+1}$ by Proposition \ref{blowup} (ii),
using the calculations of $\mathcal{L}_i|_{E_1}$ done in Lemma \ref{restrli}, and
applying the induction hypothesis to $\hat{H}_{d_1-1,d_2+1}$,
it is possible to construct all the required curves except for $C_0$ as curves lying on $E_1$.
If $d_1=2$, choose for $C_0$ any curve contracted by the natural
map $\hat{H}_{2,d_2}\to\bar{H}_2$. If $d_1>2$, choose for $C_0\subset E_2$ any curve contracted by the
natural map $E_2\simeq\hat{H}_{d_1-2,d_2+2}\times \hat{H}_{2,d_2-d_1+2}\to\hat{H}_{d_1-2,d_2+2}\times\bar{H}_2$.
\end{proof}

For $-1\leq i\leq d_1-2$, let $\hat{H}_{d_1,d_2}\to\tilde{H}^{[i]}_{d_1,d_2}$ be the contraction induced by $\mathcal{L}_i$.
For $0\leq i\leq d_1-2$, let $\hat{H}_{d_1,d_2}\to\bar{H}^{[i]}_{d_1,d_2}$ be the contraction induced by
the face of $\Nef(\hat{H}_{d_1,d_2})$ spanned by $\mathcal{L}_{i-1}$ and $\mathcal{L}_i$.
For $0\leq i\leq d_1-2$, we have natural contractions $c_i:\bar{H}^{[i]}_{d_1,d_2}\to\tilde{H}^{[i]}_{d_1,d_2}$
and $f_i:\bar{H}^{[i]}_{d_1,d_2}\to\tilde{H}^{[i-1]}_{d_1,d_2}$. 

As a particular case, $\bar{H}^{[0]}_{d_1,d_2}=\bar{H}_{d_1,d_2}$, $\tilde{H}^{[-1]}_{d_1,d_2}=\bar{H}_{d_1}$
and $f_0$ is the natural projection. 

\begin{lem}\label{stratesgen}
If $0\leq i\leq d_1-2$, the open set $\hat{H}_{d_1,d_2}\setminus\cup_{k=1}^{i+1}E_{k}$
is saturated under the contraction $\hat{H}_{d_1,d_2}\to\tilde{H}^{[i]}_{d_1,d_2}$, and its image in 
$\tilde{H}^{[i]}_{d_1,d_2}$ is an open set isomorphic to $\bar{H}_{d_1,d_2}\setminus W_{i+1}$.
\end{lem}

\begin{proof}
It suffices to show that, if $0\leq i\leq d_1-2$, a curve $C\subset\hat{H}_{d_1,d_2}$
that meets $\hat{H}_{d_1,d_2}\setminus\cup_{k=1}^{i+1}E_{k}$ is contracted by
$\hat{H}_{d_1,d_2}\to\tilde{H}^{[i]}_{d_1,d_2}$
if and only if it is contracted by $\hat{H}_{d_1,d_2}\to\bar{H}_{d_1,d_2}$.
One implication is easy: if such a curve is contracted by $\hat{H}_{d_1,d_2}\to\bar{H}_{d_1,d_2}$, since
$\mathcal{L}_i|_{\hat{H}_{d_1,d_2}\setminus\cup_{k=1}^{i+1}E_{k}}$ is the pull-back
of a line bundle on $\bar{H}_{d_1,d_2}\setminus W_{i+1}$,
it is necessarily contracted by $\hat{H}_{d_1,d_2}\to\tilde{H}^{[i]}_{d_1,d_2}$.
Let us prove the other implication by induction on $i$.
The $i=0$ case is a consequence of the description of the first contraction
$c_0=c$ in Theorem \ref{th1'}, and more precisely of the fact
that all the curves contracted by $c_0$ lie on $W_1$.
Now suppose that $i>0$, and let $C$ be a curve meeting
$\hat{H}_{d_1,d_2}\setminus\cup_{k=1}^{i+1}E_{k}$ that is contracted by
$\hat{H}_{d_1,d_2}\to\tilde{H}^{[i]}_{d_1,d_2}$. In particular, $C\cdot E_{i+1}\geq 0$ and $C\cdot \mathcal{L}_i=0$.
From the identity $(\mathcal{L}_{i+1}-\mathcal{L}_i)(E_{i+1})\simeq(\mathcal{L}_i-\mathcal{L}_{i-1})$ and the fact
that $\mathcal{L}_{i-1}$ and $\mathcal{L}_{i+1}$ are semi-ample, hence nef, we see that we necessarily have
$C\cdot \mathcal{L}_{i-1}=0$. Consequently, $C$ is contracted by 
$\hat{H}_{d_1,d_2}\to\tilde{H}^{[i-1]}_{d_1,d_2}$ and the induction hypothesis applies to show that
$C$ is contracted by $\hat{H}_{d_1,d_2}\to\bar{H}_{d_1,d_2}$.
\end{proof}

\begin{prop}\label{strates}
\hspace{1em}

\begin{enumerate}
\item[(i)] If $-1\leq i\leq d_1-2$, the scheme $\tilde{H}^{[i]}_{d_1,d_2}$
admits a stratification by $i+2$ locally closed subschemes whose
normalized strata are $(\bar{H}_{d_1-i+r,d_2+i-r}\setminus W_{r+1})_{0\leq r\leq i}$ and $\bar{H}_{d_1-i-1}$.
\item[(ii)] If $0\leq i\leq d_1-2$, the scheme $\bar{H}^{[i]}_{d_1,d_2}$
admits a stratification by $i+1$ locally closed subschemes whose
normalized strata are $(\bar{H}_{d_1-i+r,d_2+i-r}\setminus W_{r})_{0\leq r\leq i}$.
\end{enumerate}
\end{prop}

\begin{proof}
Let us prove (i), the proof of (ii) being analogous.
If $i=-1$, we know that $\tilde{H}^{[-1]}_{d_1,d_2}\simeq\bar{H}_{d_1}$, so that we may suppose that $i\geq 0$.

By Lemma \ref{stratesgen}, the open set $\hat{H}_{d_1,d_2}\setminus\cup_{k=1}^{i+1}E_{k}$
is saturated under the contraction $\hat{H}_{d_1,d_2}\to\tilde{H}^{[i]}_{d_1,d_2}$, and its image in 
$\tilde{H}^{[i]}_{d_1,d_2}$ is an open set isomorphic to $\bar{H}_{d_1,d_2}\setminus W_{i+1}$.

Using the descriptions of $E_j|_{E_k}$ and of $\mathcal{L}_i|{E_k}$
from Proposition \ref{blowup} and Lemma \ref{restrli}, we see
successively for $1\leq k\leq i$ (applying Lemma \ref{stratesgen} to
$\hat{H}_{d_1-k,d_2+k}\to\tilde{H}^{[i-k]}_{d_1-k,d_2+k}$) 
that $E_k\setminus \cup_{j=k+1}^{i+1}(E_j|_{E_k})$ is saturated under
the contraction $\hat{H}_{d_1,d_2}\to\tilde{H}^{[i]}_{d_1,d_2}$, and that its image in 
$\tilde{H}^{[i]}_{d_1,d_2}$ is a locally closed subscheme isomorphic (up to normalization) to
$\bar{H}_{d_1-k,d_2+k}\setminus W_{i-k+1}$.

As $E_{i+1}$ is the complement of all the saturated subsets already described, it is also saturated.
The description of $\mathcal{L}_i|{E_{i+1}}$
in Lemma \ref{restrli} shows that its image in 
$\tilde{H}^{[i]}_{d_1,d_2}$ is a closed subscheme isomorphic (up to normalization) to $\bar{H}_{d_1-i-1}$.
\end{proof}

\pagebreak[1]

\begin{prop}\label{Qfact}
\hspace{1em}

\begin{enumerate}
\item[(i)] If $-1\leq i\leq d_1-2$, $\tilde{H}^{[i]}_{d_1,d_2}$ has Picard rank $1$, and if
$0\leq i\leq d_1-2$, $\bar{H}^{[i]}_{d_1,d_2}$ has Picard rank $2$.
\item[(ii)] If $0\leq i\leq d_1-2$, the contractions $\hat{H}_{d_1,d_2}\to\tilde{H}^{[i]}_{d_1,d_2}$
and $\hat{H}_{d_1,d_2}\to\bar{H}^{[i]}_{d_1,d_2}$ are birational.
\item[(iii)] 
If $0\leq i\leq d_1-2$, $\bar{H}^{[i]}_{d_1,d_2}$ is isomorphic to $\bar{H}_{d_1,d_2}$ in
codimension $1$, and $\mathbb{Q}$-factorial. Its nef cone is generated by the
semi-ample line bundles $\mathcal{L}_{i-1}$ and $\mathcal{L}_i$. 
\item[(iv)] If $0\leq i\leq d_1-3$, $c_i$ is a small contraction, and $f_{i+1}$ is also a small contraction: its flip.
\item[(v)] The contraction $c_{d_1-2}$ is a divisorial contraction contracting the discriminant.
\end{enumerate}
\end{prop}

\begin{proof}
Part (i) is a consequence of the dimensions of the faces of $\Nef(\hat{H}_{d_1,d_2})$ used to construct
$\tilde{H}^{[i]}_{d_1,d_2}$ and $\bar{H}^{[i]}_{d_1,d_2}$. Part
(ii) and the first assertion of Part (iii) are corollaries of Proposition \ref{strates}. 
Since $\bar{H}_{d_1,d_2}\dashrightarrow\bar{H}^{[i]}_{d_1,d_2}$ is an isomorphism in codimension $1$ between two varieties
of Picard rank $2$, and since $\bar{H}_{d_1,d_2}$ is $\mathbb{Q}$-factorial (because
it is smooth), it follows that $\bar{H}^{[i]}_{d_1,d_2}$ is $\mathbb{Q}$-factorial.
Moreover, the semi-ample line bundles
$\mathcal{L}_{i-1}$ and $\mathcal{L}_i$ induce the contractions $f_i$ and $c_i$, hence are
on the boundary of the nef cone of $\bar{H}^{[i]}_{d_1,d_2}$.
Since $\bar{H}^{[i]}_{d_1,d_2}$ is of Picard rank $2$,
they generate its nef cone.

Part (iv) is an immediate consequence of the fact proven in (ii) that $\bar{H}^{[i]}_{d_1,d_2}$ and $\bar{H}^{[i+1]}_{d_1,d_2}$
are isomorphic in codimension $1$.
By Proposition \ref{strates}, the discriminant is not contracted in $\bar{H}^{[d_1-2]}_{d_1,d_2}$,
but is contracted in $\tilde{H}^{[d_1-2]}_{d_1,d_2}$, proving (v).
\end{proof}

It is now possible to prove Theorem \ref{th3}.

\begin{thm}[Theorem \ref{th3}]\label{th3'}
\hspace{1em}

\begin{enumerate}
\item[(i)] The variety $\bar{H}_{d_1,d_2}$ is a
Mori dream space and its effective cone is generated by $\mathcal{O}(1,0)$ and $\Delta$. 
\item[(ii)] The MMP for $\bar{H}_{d_1,d_2}$ flips the
$W_i$ for $1\leq i\leq d_1-2$ and eventually contracts $W_{d_1-1}=\Delta$.
\item[(iii)] The last model of the MMP for $\bar{H}_{d_1,d_2}$
is a compactification of $H_{d_1,d_2}$ with codimension $2$ boundary, that admits a stratification
whose normalized strata are $(H_{d_1-i,d_2+i})_{0\leq i\leq d_1-1}$.
\end{enumerate}
\end{thm}

\begin{proof}
By Proposition \ref{Qfact} (iii), the $\bar{H}^{[i]}_{d_1,d_2}$
are small $\mathbb{Q}$-factorial modifications of $\bar{H}_{d_1,d_2}$.
The variety $\bar{H}_{d_1,d_2}$ is the total space of the fibration $\bar{H}_{d_1,d_2}\to\bar{H}_{d_1}$ and,
by Proposition \ref{Qfact} (iv) and (v), after performing a sequence of flips
leading to $(\bar{H}^{[i]}_{d_1,d_2})_{0\leq i\leq d_1-2}$, the total space of the divisorial contraction $c_{d_1-2}$.
This implies that the nef cones of the
$(\bar{H}^{[i]}_{d_1,d_2})_{0\leq i\leq d_1-2}$
cover the movable cone of $\bar{H}_{d_1,d_2}$. Moreover, for $0\leq i\leq d_1-2$,
the nef cone of $\bar{H}^{[i]}_{d_1,d_2}$ is generated by semi-ample line bundle
by Proposition \ref{Qfact} (iii). We have checked the hypotheses of Definition 1.10 of \cite{MDS}, proving that 
$\bar{H}_{d_1,d_2}$ is a Mori dream space.

Moreover, the existence of the fibration $\bar{H}_{d_1,d_2}\to\bar{H}_{d_1}$
(resp. of the divisorial contraction $c_{d_1-2}$) show that $\mathcal{O}(0,1)$ (resp. $\Delta$) are on the boundary
of the effective cone of $\bar{H}_{d_1,d_2}$. Since $\bar{H}_{d_1,d_2}$ has Picard rank $2$, they generate it, proving (i).

The description in (ii) of the flipped loci follow from the explicit description of $\bar{H}^{[i]}_{d_1,d_2})$ in Proposition
\ref{strates} (ii). The last model of the MMP for $\bar{H}_{d_1,d_2}$ is $\tilde{H}^{[d_1-2]}_{d_1,d_2}$ and is described
in Proposition \ref{strates} (i).
\end{proof}

\begin{rem}
From the explicit descriptions of the $\bar{H}^{[i]}_{d_1,d_2}$ in Proposition \ref{strates} (ii),
it is possible to understand (up to normalization) what happens to $W_k$ during the
MMP for $\bar{H}_{d_1,d_2}$. At the beginning, $W_k$ is isomorphic (up to normalization) to
$\bar{H}_{d_1-k}\times\bar{H}_{k,d_2-d_1+k}$. During the first $k-1$ flips, it follows the MMP for $\bar{H}_{k,d_2-d_1+k}$:
in particular, after the $(k-1)^{\thh}$, it becomes isomorphic (up to normalization) to
$\bar{H}_{d_1-k}\times\tilde{H}^{[k-1]}_{k,d_2-d_1+k}$ During the $k^{\thh}$ flip, it is contracted via
$\bar{H}_{d_1-k}\times\tilde{H}^{[k-1]}_{k,d_2-d_1+k}\to\bar{H}_{d_1-k}$, and then flipped via
$\bar{H}_{d_1-k,d_2+k}\to\bar{H}_{d_1-k}$. Then, during the last $d_1-k-1$ steps, it follows
the MMP for $\bar{H}_{d_1-k,d_2+k}$. In particular, in the last model, it becomes isomorphic (up to normalization) to
$\tilde{H}^{[d_1-k-2]}_{d_1-k,d_2+k}$.
\end{rem}

\subsection{Complete families}

As a consequence of Proposition \ref{strates}, we construct complete curves in $H_{d_1,d_2}$.

\begin{prop}[Proposition \ref{thcc}]\label{thccbis}
The variety $H_{d_1,d_2}$ contains complete curves.
\end{prop}

\begin{proof}
By Proposition \ref{strates}, $\tilde{H}^{[d_1-2]}_{d_1,d_2}$ is a projective
compactification of $H_{d_1,d_2}$ with codimension $2$ boundary.
Taking general hyperplane sections, we construct a complete curve in
$\bar{H}^{[d_1-2]}_{d_1,d_2}$ that avoids the boundary, that is a complete curve
in $H_{d_1,d_2}$.
\end{proof}

\begin{rem}\label{explicurve}
In general, I do not know how to construct such curves by hand,
without using the compactification $\tilde{H}^{[d_1-2]}_{d_1,d_2}$. However, there are
particular cases for which it is possible.

When $d_2=d_1+1$, an explicit complete curve in $H_{d_1,d_1+1}$ is induced by:
$$t\mapsto[X_0^{d_1}+tX_0^{d_1-1}X_1+\dots+t^{d_1}X_1^{d_1},
X_0X_1(X_0^{d_1-1}+tX_0^{d_1-2}X_1+\dots+t^{d_1-1}X_1^{d_1-1})].$$

In characteristic $p$, it is possible to use $p^{\thh}$-powers.
For instance, there is a well-defined map $\psi_p:\bar{H}_{1,d_2}\to\bar{H}_{p,pd_2}$ given
by $\psi_p([F,G])=[F^p,G^p]$. Its image is a complete curve in $H_{p,pd_2}$.
\end{rem}

\begin{rem}
One reason why it is difficult to answer Question \ref{q2} when, say, $N=3$ and $d_1\geq 2$,
is that such curves cannot be rational (as there are no non-isotrivial smooth families of curves over $\mathbb{P}^1$).
When $N=1$, I do not know if there is an analogous obstruction for some values of the degrees,
or if it is always possible to construct complete rational curves in $H_{d_1,d_2}$.
\end{rem}

\subsection{The Hilbert scheme}\label{idhilb}

In this last paragaph, we will give an interpretation of $\hat{H}_{d_1,d_2}$ as a multigraded Hilbert scheme \cite{multigraded}.
Combined with Proposition \ref{nef}, this will prove Proposition \ref{thnef}.

Consider the natural action of $\mathbb{G}_m$ on $\mathbb{A}^2$ by homotheties.
If $Z\subset\mathbb{A}^2$ is a $\mathbb{G}_m$-invariant closed subscheme, its Hilbert function
$\HF_Z(l)$ is the dimension of the subspace of $S_l=H^0(\mathbb{P}^1,\mathcal{O}(l))$
consisting of equations satisfied by $Z$. Note that our convention is different from \cite{multigraded},
that considers the dimension of the quotient: it will be more convenient for us to manipulate equations of $Z$ rather
than functions on $Z$.

\begin{lem}\label{HF}
 Let $[F,G]\in H_{d_1,d_2}$. Then:

$$
\HF_{\{F=G=0\}}(l)=\left\{
    \begin{array}{llll}
       0&\textrm{ if } \hspace{0.3em}l\leq d_1-1, \\
       l-d_1+1&\textrm{ if } \hspace{0.3em}d_1\leq l\leq d_2-1, \\
       2l-d_1-d_2+2&\textrm{ if } \hspace{0.3em}d_2\leq l\leq d_1+d_2-1, \\
       l+1&\textrm{ if } \hspace{0.3em}d_1+d_2\leq l. 
    \end{array}
\right.
$$
We will denote by $\HF_{d_1,d_2}(l)$ this function.
\end{lem}

\begin{proof}
If $l\leq d_1-1$, there are obviously no non-zero equations. If $d_1\leq l\leq d_2-1$, there are only the multiples of $F$.
If $d_2\leq l\leq d_1+d_2-1$, there are the multiples of $F$ and the multiples of $G$. Since $F$ and $G$ have no common factor,
those two subspaces have trivial intersection, and they are in direct sum. If $l=d_1+d_2$, however,
the intersection of the multiples of $F$ and of the multiples of $G$ is one-dimensional, generated by $FG$.
It follows that $HF_{\{F=G=0\}}(d_1+d_2)=d_1+d_2+1$, hence that $\{F=G=0\}$ satisfies every degree $d_1+d_2$ equation.
As a consequence, $\{F=G=0\}$ satisfies every degree $l$ equation for $l\geq d_1+d_2$ .
\end{proof}

Let $\Hilb_{d_1,d_2}$ be the multigraded Hilbert scheme of $\mathbb{G}_m$-invariant subschemes 
of $\mathbb{A}^2$ with Hilbert function $\HF_{d_1,d_2}$,
as defined and constructed in \cite{multigraded} Theorem 1.1. 
In the sequel, we will always use the same notation for a subscheme of $\mathbb{A}^2$ and a point it induces on a Hilbert scheme.

\begin{prop}\label{hilbhilb}
The scheme $\Hilb_{d_1,d_2}$ is naturally a projective subscheme of the Hilbert scheme $\Hilb_{\mathbb{P}^2}$ of $\mathbb{P}^2$.
It  is a smooth compactification of $H_{d_1,d_2}$, and
there exists a compatible birational morphism
$\pi: \Hilb_{d_1,d_2}\to\bar{H}_{d_1,d_2}$.
\end{prop}

\begin{proof}
The scheme $\Hilb_{d_1,d_2}$ is projective by \cite{multigraded} Corollary 1.2. The subschemes parametrized by $\Hilb_{d_1,d_2}$ 
satisfy all equations of degree $\geq d_1+d_2$, hence are set-theoretically supported on the origin.
Thus, they may be viewed as closed subschemes of $\mathbb{P}^2$.
The induced
natural transformation $\Hilb_{d_1,d_2}\to\Hilb_{\mathbb{P}^2}$ is a monomorphism,
as one sees from the description of the functors of points of these two schemes. By \cite{EGA44} Corollaire 18.12.6, it is a closed immersion.
Moreover, $\Hilb_{d_1,d_2}$ is smooth and connected by
\cite{Evain} Theorem 1 (see also the more general results of \cite{SmithMaclagan} Theorem 1.1).

If $Z\in \Hilb_{d_1,d_2}$, by the choice of $\HF_{d_1,d_2}$, $Z$ satisfies a unique degree $d_1$ equation $F$ up to scalar multiple,
and a unique degree $d_2$ equation $G$ up to scalar multiple and up to a multiple of $F$. This induces a morphism
$\pi: \Hilb_{d_1,d_2}\to\bar{H}_{d_1,d_2}$ given by $\pi(Z)=[F,G]$.
Of course, if $\pi(Z)=[F,G]\in H_{d_1,d_2}$, we necessarily have
$Z=\{F=G=0\}$, as we have an inclusion and the
spaces of degree $l$ equations of $Z$ and $\{F=G=0\}$
have the same dimension for any $l$. 

On the other hand, there is a natural section of $\pi$ above $H_{d_1,d_2}$ given by
$[F,G]\mapsto \{F=G=0\}$. It follows that $\pi$ is an isomorphism above $H_{d_1,d_2}$.
Hence, $\pi$ is birational, and $\Hilb_{d_1,d_2}$ is a smooth compactification of $H_{d_1,d_2}$.
\end{proof}

The goal of this paragraph is to prove that $\Hilb_{d_1,d_2}$ coincides with $\hat{H}_{d_1,d_2}$. A natural way to do it would be to
construct the universal family over $\hat{H}_{d_1,d_2}$. We do not know how to do it directly, and use our knowledge of the
geometry of $\hat{H}_{d_1,d_2}$ instead.

\begin{lem}\label{excexc}
Let $1\leq k\leq d_1-1$. There exists an injective morphism $$e_k:\Hilb_{d_1-k,d_2+k}\times\Hilb_{k,d_2-d_1+k}\to\Hilb_{d_1,d_2}$$
satisfying: if $Z\in\Hilb_{d_1-k,d_2+k}$, $W\in\Hilb_{k,d_2-d_1+k}$, and $\pi(Z)=[F,G]$,
$e_k(Z,W)$ is defined by the equations of the form $FK$ for any equation $K$ of $W$,
and by the equations of $Z$ of degrees $\geq d_2+k$.

The irreducible components of the complement of $H_{d_1,d_2}$ in $\Hilb_{d_1,d_2}$ are exactly the divisors $\Ima(e_k)$.

The natural rational map $r:\hat{H}_{d_1,d_2}\dashrightarrow \Hilb_{d_1,d_2}$ induces by restriction to $E_k$ the natural
rational map $\hat{H}_{d_1-k,d_2+k}\times\hat{H}_{k,d_2-d_1+k} \dashrightarrow\Hilb_{d_1-k,d_2+k}\times\Hilb_{k,d_2-d_1+k}$.
\end{lem}

\begin{proof}
Let us first show that $e_k$ is well-defined. To do so, fix $Z\in\Hilb_{d_1-k,d_2+k}$ and $W\in\Hilb_{k,d_2-d_1+k}$, and
write $\pi(Z)=[F,G]$. Let $Y$ be the subscheme defined by equations of the form $FK$ for any equation $K$ of $W$.
It is easy to describe $\HF_Y$ from $\HF_W$, that is known by Lemma \ref{HF}.
Since $e_k(Z,W)$ is defined by additional equations
of degrees $\geq d_2+k$, it follows that $\HF_{e_k(Z,W)}$ coincides with $\HF_{d_1,d_2}$
for $l<d_2+k$. Moreover, again by Lemma \ref{HF}, the equations of degrees $\geq d_2+k$ of $Y$ are exactly the
multiples of $F$, hence are also equations of $Z$. It follows that the equations of degrees $\geq d_2+k$ of
$Z$ and $e_k(Z,W)$ are the same. Since $\HF_Z$ is known by Lemma \ref{HF}, one checks that 
$\HF_{e_k(Z,W)}$ coincides with $\HF_{d_1,d_2}$
for $l\geq d_2+k$. We have proven as wanted that
$e_k(Z,W)\in\Hilb_{d_1,d_2}$, hence that $e_k$ is well-defined.

It is easy to see from the above construction that $e_k$ is injective. Indeed, $F$ is recovered as the
greatest common divisor of the equations of 
$e_k(Z,W)$ of degrees $<d_2+k$, the equations of $W$ are recovered by dividing these equations by $F$, and the additional
equations of $Z$ are recovered as they are the equations of $e_k(Z,W)$ of degrees $\geq d_2+k$.

By injectivity of $e_k$, a dimension computation shows that $\Ima(e_k)$ is a divisor in $ \Hilb_{d_1,d_2}$.
It is easily checked that $\pi(\Ima(e_k))=W_k$: this shows that these divisors are distinct and do not
meet $H_{d_1,d_2}$. Let us show conversely that if $Y\in \Hilb_{d_1,d_2}\setminus H_{d_1,d_2}$, $Y$ is included in one of these divisors.
Let $k$ be such that $\pi(Y)\in W_k\setminus W_{k-1}$, and write  $\pi(Y)=[PL,PH]$ with $\deg(P)=d_1-k$.
Set $W=\{L=H=0\}$, and let $Z$ be the subscheme defined by $P$ and by all the equations of $Y$ of degrees $\geq d_2+k$.
It is straightforward to check that $Z\in \Hilb_{d_1-k,d_2+k}$, $W\in\Hilb_{k,d_2-d_1+k}$
and $Y=e_k(Z,W)$.

It remains to prove the last assertion. The natural rational map $r:\hat{H}_{d_1,d_2}\dashrightarrow \Hilb_{d_1,d_2}$ is defined
on an open set whose complement has codimension $\geq 2$ because  $\hat{H}_{d_1,d_2}$ is smooth and $ \Hilb_{d_1,d_2}$
is proper. This set of definition intersects the divisor $E_k$. Let $x=([P,S],[L,H])\in  H_{d_1-k,d_2+k}\times H_{k,d_2-d_1+k}$
be a general point of $E_k$ included in this locus of definition. By Remark \ref{identbir}, there exist
$F\in S_{d_1}$ and $G\in S_{d_2}$ such that $x=([P,LG-HF],[L,H])$ and $x=\lim_{t\to 0}[PL+tF,PH+tG]$
in $\hat{H}_{d_1,d_2}$. On the other hand, $\lim_{t\to 0}[PL+tF,PH+tG]$ in $\Hilb_{d_1,d_2}$ satisfies the equations
$PL$, $PH$ and $S=LG-HF$: it is included in, hence equal to $e_k([P,S],[L,H])$. This ends the proof of the lemma.
\end{proof}

It is now possible to conclude:

\begin{prop}
The rational map  $r:\hat{H}_{d_1,d_2}\dashrightarrow \Hilb_{d_1,d_2}$ is an isomorphism.
\end{prop}

\begin{proof}
By Lemma \ref{excexc}, we know that $r$ is an isomorphism in codimension $1$.
Let us denote by $U$ the biggest open subset over which $r$ is an isomorphism: its complement has codimension $\geq 2$ in both
$\hat{H}_{d_1,d_2}$ and $\Hilb_{d_1,d_2}$. Since $\hat{H}_{d_1,d_2}$ and $\Hilb_{d_1,d_2}$
are smooth, their Picard groups are identified with $\Pic(U)$. We will construct a line bundle $\mathcal{L}$ on $U$
that is ample on both $\hat{H}_{d_1,d_2}$ and $\Hilb_{d_1,d_2}$. This will prove the assertion,
because $\hat{H}_{d_1,d_2}$ and $\Hilb_{d_1,d_2}$ will be both isomorphic to $\Proj\bigoplus_{k\geq 0}H^0(U,\mathcal{L}^{\otimes k})$,
as they are normal.

Let $\Hilb_{\mathbb{P}^2}^P$ be the connected component of $\Hilb_{\mathbb{P}^2}$  containing $\Hilb_{d_1,d_2}$
as in  Proposition \ref{hilbhilb}. By Grothendieck's construction
of the Hilbert scheme as a subscheme of a Grassmannian, if  $d\gg 0$ and
$\mathcal{F}_d$ is the tautological subbundle of
$H^0(\mathbb{P}^2,\mathcal{O}(d))$ on $\Hilb_{\mathbb{P}^2}^P$,
$\det(\mathcal{F}_d)^{-1}$ is ample on $\Hilb_{\mathbb{P}^2}^P$.
As a consequence, its res\-triction $\mathcal{L}$ is ample on $\Hilb_{d_1,d_2}$.
Consider $S_l$ as a constant vector bundle on $U$, and let $\mathcal{E}_l\subset S_l$ be the tautological subbundle.
Notice that $\mathcal{F}_d|_{\Hilb_{d_1,d_2}}$
splits as a direct sum of eigenspaces with respect to the $\mathbb{G}_m$-action, inducing an isomorphism
$\mathcal{F}_d|_{\Hilb_{d_1,d_2}}\simeq\bigoplus_{l=0}^d\mathcal{E}_l$.
Consequently, $\mathcal{L}\simeq\bigotimes_{l=0}^d\det(\mathcal {E}_l)^{-1}$.

By Lemma \ref{HF}, if $l<d_1$, $\mathcal{E}_l=0$,
and if $l\geq d_1+d_2$, $\mathcal{E}_l=S_l$. In both cases, $\det(\mathcal{E}_l)\simeq\mathcal{O}$.
By Lemma \ref{HF}, if $d_1\leq l\leq d_2-1$, there is an isomorphism $S_{l-d_1}(-1,0)\simeq \mathcal{E}_l$ given by multiplication by $F$. 
It follows that $\det(\mathcal{E}_l)\simeq\mathcal{O}(-(l-d_1+1),0)$. If $d_2\leq l\leq d_1+d_2-1$,
there is a morphism of coherent sheaves $S_{l-d_1}(-1,0)\oplus S_{l-d_2}(0,-1)\to \mathcal{E}_l$ given by multiplication by
$F$ and $G$. This morphism is an isomorphism over $H_{d_1,d_2}$, by Lemma \ref{HF}. In particular, it is injective, and its cokernel
$\mathcal{Q}$ is set-theoretically included in the union of the exceptional divisors.
Lemma \ref{multexc} describes $\mathcal{Q}$ in a neighbourhood of the generic points of the exceptional divisors, allowing to compute that
$\det(\mathcal{E}_l)\simeq \mathcal{O}(-(l-d_1+1),-(l-d_2+1))(\sum_{k=1}^{l-d_2}(l-d_2-k+1)E_k)$.
We recognize: $\det(\mathcal{E}_l)\simeq \mathcal{L}_{l-d_2}^{-1}$.

Taking into account the fact that
$\mathcal{L}_{d_1-1}$ is trivial by Lemma \ref{restrli} (i), one obtains that 
$\mathcal{L}\simeq\mathcal{L}_{-1}^{\otimes \frac{(d_2-d_1)(d_2-d_1+1)}{2}}\otimes\bigotimes_{i=0}^{d_1-2}\mathcal{L}_i$
is ample on $\Hilb_{d_1,d_2}$ (and independent of $d\geq d_1+d_2-1$).
On the other hand, $\mathcal{L}$ is in the interior of the nef cone of $\hat{H}_{d_1,d_2}$ by Proposition \ref{nef}.
Hence, it is also ample on $\hat{H}_{d_1,d_2}$ by Kleiman's criterion (see \cite{Lazarsfeld1} Theorem 1.4.23). This concludes the proof.
\end{proof}

We needed the following lemma:

\begin{lem}\label{multexc}
Let $d_2\leq l\leq d_1+d_2-1$, let $\mathcal{Q}$ be the cokernel of the
morphism of coherent sheaves $S_{l-d_1}(-1,0)\oplus S_{l-d_2}(0,-1)\to \mathcal{E}_l$ on $U$ given by multiplication by
$F$ and $G$, and let $1\leq k\leq d_1-1$.

If $k>l-d_2$, $\mathcal{Q}$ is trivial in a neighbourhood of the generic point of $E_k$.

If $k\leq l-d_2$, $\mathcal{Q}$ is a rank $l-d_2-k+1$ vector bundle on $E_k$
in a neighbourhood of the generic point of $E_k$,
\end{lem}

\begin{proof}
Let $x=([P,S],[L,H])\in  H_{d_1-k,d_2+k}\times H_{k,d_2-d_1+k}$
be a general point of $E_k\cap U$. By Remark \ref{identbir}, there exist
$F\in S_{d_1}$ and $G\in S_{d_2}$ such that $x=([P,LG-HF],[L,H])$ and $x=\lim_{t\to 0}[PL+tF,PH+tG]$
in $\hat{H}_{d_1,d_2}$.
Consider the morphism $i:\Spec(\mathbbm{k}[[t]])\to U$ given by $t\mapsto[PL+tF,PH+tG]$. By base change, we get morphisms of
$\mathbbm{k}[[t]]$-modules $i^*S_{l-d_1}\oplus i^* S_{l-d_2}\to i^*\mathcal{E}_l\subset i^*S_l$. Note that $i^*\mathcal{E}_l$ 
is still a subbundle of $i^*S_l$ by flatness of $\mathcal{E}_l$, that $i^*\mathcal{Q}$ is the cokernel of 
$i^*S_{l-d_1}\oplus i^* S_{l-d_2}\to i^*\mathcal{E}_l$ by right-exactness of tensor product,
and hence that  $i^*\mathcal{Q}$ is the torsion submodule of the cokernel of the morphism
 $i^*S_{l-d_1}\oplus i^* S_{l-d_2}\to i^*S_l$ given by $(A,B)\mapsto A(PL+tF)+B(PH+tG)$.

Let us first compute $\mathcal{Q}_x=(i^*\mathcal{Q})_0$: it is the cokernel of $S_{l-d_1}\oplus S_{l-d_2}\to (\mathcal{E}_l)_0$
given by $(A,B)\mapsto APL+BPH$. Since $S_{l-d_1}\oplus S_{l-d_2}$ and $\mathcal{E}_l$ have the same rank,
it has the same dimension as the kernel of
$(A,B)\mapsto APL+BPH$. This kernel is easy to compute (as $L$ is prime to $H$):
it has dimension $0$ if $k> l-d_2$ and dimension $l-d_2-k+1$ if $k\leq l-d_2$.

It remains to show that the scheme-theoretical support of $\mathcal{Q}$ in a neighbourhood of the generic point of $E_k$
is included in $E_k$ with its reduced structure. By compatibility of taking the support with base-change,
it suffices to show that the support of $i^*\mathcal{Q}$ is included in the reduced origin
of $\Spec(\mathbbm{k}[[t]])$. To do so, one needs to prove that if
$T$ is a section of $i^*\mathcal{Q}$ such that $t^2T=0$, then $tT=0$.
This boils down to proving that if $A\in i^*S_{l-d_1}=S_{l-d_1}\otimes_k \mathbbm{k}[[t]]$
and $B\in i^* S_{l-d_2}=S_{l-d_2}\otimes_k \mathbbm{k}[[t]]$ are such
that $t^2|A(PL+tF)+B(PH+tG)$, then $t|A$ and $t|B$. Let us introduce $A_0,A_1\in S_{l-d_1}$
and $B_0,B_1\in i^* S_{l-d_2}$ the terms of $A,B$ of order $0$ and $1$ in $t$. The hypothesis means that
$A_0PL+B_0PH=0$ and $P(A_1L+B_1H)+A_0F+A_1G=0$. From the first equation and because $L$ is prime to $H$,
we see that it is possible to write $A_0=CH$ and $B_0=-CL$.
Consequently, one sees from the second equation and because $P$ is prime to $S$ that $P|C$. For degree reasons, $C=0$, hence $A_0=B_0=0$ as wanted.
\end{proof}


\bibliographystyle{plain}
\bibliography{biblio}

\end{document}